\documentclass[12pt]{amsart}
\usepackage{amssymb,amsfonts,amsthm, algorithm, algpseudocode, color}
\usepackage{graphicx}
\usepackage{bm}
\newtheorem{thm}{Theorem}[section]

\newtheorem{lem}[thm]{Lemma}

\theoremstyle{definition}
\newtheorem{de}[thm]{Definition}
\theoremstyle{remark}

\numberwithin{equation}{section}

\begin{document}
\title[Estimation of sample quality for SDE]{Using coupling methods to estimate sample quality
for stochastic differential equations}
\author{Matthew Dobson}
\address{Matthew Dobson: Department of Mathematics and Statistics, University
  of Massachusetts Amherst, Amherst, MA, 01002, USA}
\email{dobson@math.umass.edu}
\author{Yao Li}
\address{Yao Li: Department of Mathematics and Statistics, University
  of Massachusetts Amherst, Amherst, MA, 01002, USA}
\email{yaoli@math.umass.edu}
\author{Jiayu Zhai}
\address{Jiayu Zhai: Department of Mathematics and Statistics, University
  of Massachusetts Amherst, Amherst, MA, 01002, USA}
\email{zhai@math.umass.edu}

\thanks{Yao Li is partially supported by NSF DMS-1813246.}


\begin{abstract}
  A probabilistic approach for estimating sample qualities for stochastic
  differential equations is introduced in this paper. The aim is to
  provide a quantitative upper bound of the distance between the
  invariant probability measure of a stochastic differential
  equation and that of its numerical approximation. In order to extend
  estimates of finite time truncation error to infinite time, it is
  crucial to know the rate of contraction of the transition kernel of
  the SDE. We find that suitable numerical coupling methods can effectively
  estimate such rate of contraction, which gives the distance between two
  invariant probability measures. Our algorithms are tested with
  several low and high dimensional numerical examples.
\end{abstract}

\maketitle

\section{Introduction}
Stochastic differential equations (SDEs) are widely used in many
scientific fields. Under mild assumptions, an SDE would admit a unique
invariant probability measure, denoted by $\pi$. In many applications
including but not limited to Markov chain Monte Carlo and molecular dynamics, it is important
to sample from $\pi$ \cite{andrieu2003introduction, lelievre2016partial}. This is
usually done by either numerically integrating an SDE over a very long
trajectories or integrating many trajectories of the SDE over a finite
time \cite{milstein2007computing}. However, a numerical integrator of
the SDE typically has a different invariant probability measure, 
denoted by $\hat{\pi}$, that depends on the time discretization \cite{talay1990second,
  talay1990expansion, roberts1996exponential}. A natural question is 
that, how is $\hat{\pi}$ different from $\pi$? In other words, what is the
quality of data sampled from a numerical trajectory of the SDE? 
This is very different from the classical truncation error analysis,
which is only applicable for finite time intervals except some special
cases \cite{leimkuhler2013rational,chen2017approximation}. 

Theoretically, it is well known that the distance between $\pi$ and
$\hat{\pi}$ can be controlled if we have good estimate of (i) the finite
time truncation error and (ii) the rate of geometric ergodicity of
the SDE.  Estimates of this type can be made by various different
approaches \cite{mattingly2002ergodicity, mattingly2010convergence,
  bou2018continuous, bally1996law, bally1996law2}. Roughly speaking, if the truncation error
over a finite time interval $[0, T]$ is $O(\epsilon)$, and the rate of geometric
ergodicity is $\gamma$ (i.e. speed of convergence to $\pi$ is
$\approx \gamma^{t}$ for $\gamma \in (0, 1)$), then the difference between $\pi$ and $\hat{\pi}$ is
$O(\epsilon(1 - \gamma^{T})^{-1})$. (See our discussion in Section 3.1
for details.) However, these approaches can not give a
quantitative estimate in general, as the rate of geometric ergodicity
$\gamma$ estimated by rigorous approaches are usually very far from
being sharp. Many approaches such as the Lyapunov function method can
only rigorously show that the speed of convergence is 
$\approx \gamma^{t}$ for {\it some} $\gamma < 1$ \cite{meyn2012markov,
hairer2010convergence, hairer2011yet}. Looking into
the proof more carefully, one can easily find that this $\gamma$ has
to be extremely close to $1$ to make the proof work. This gives a
very large $(1 - \gamma^{T})^{-1}$ and makes rigorous estimates difficult
to use in practice. To the best of our knowledge, quantitative
estimates of convergence rate can only be proved for a few special
cases like stochastic gradient flow and Langevin dynamics 
\cite{eberle2019couplings, bou2018coupling, bakry1985diffusions}.

The aim of this paper is to provide some algorithms to numerically
estimate the distance between $\pi$ and $\hat{\pi}$. The finite time
truncation error over a time interval $[0, T]$ is estimated by using extrapolations, which is a
common practice in numerical analysis. The main novel part is the
estimation of the rate of contraction of the transition kernel. 
Traditional approaches
for computing the rate of geometric ergodicity are either
computing principal eigenvalue of the discretized generator or
estimating the decay rate of correlation. The eigenvalue method works
well in low dimension but faces significant challenge if the SDE is in
dimension $\geq 3$. The correlation decay is difficult to estimate as
well, because a correlation has exponentially small expectation and
large variance. One needs a huge amount of samples to estimate it
effectively. In addition, exponential decay of correlation with
respect to an ad-hoc observable is usually not very convincing. In
this paper, we propose to estimate the rate of contraction of the
transition kernel by using a coupling technique. 

Coupling methods have been used in rigorous proofs for decades
\cite{lindvall2002lectures, lindvall1986coupling,
  eberle2011reflection, mufa1996estimation}. The idea
is to run two trajectories of a random process $X_{t}$, such that one is from a given initial distribution and
the other is stationary. A suitable joint distribution, called a coupling, is constructed
in the product space, such that two marginals of this joint process
are the original two trajectories. If after some time, the two processes
stay together with high probability, then the law of $X_{t}$ must
be very close to its invariant probability measure. It is well known that the coupling lemma gives bounds of both total
variation norm and some 1-Wasserstein-type distances. In this paper,
we use the coupling method numerically. If two numerical trajectories meet
each other, they are coupled and evolve together after
coupling. By the coupling lemma, the contraction rate of the
transition kernel can be estimated
numerically by computing the probability of successful coupling, which
follows from running a Monte Carlo simulation. Together with the finite time error, we can estimate the
distance between $\pi$ and $\hat{\pi}$. The main advantage of coupling
method is that it is relatively dimension-free, and we demonstrate  
our technique on an SDE system in $\mathbb{R}^{80}$ in Section~\ref{sec:fitzhugh}.

We provide two sets of algorithms, one for a quantitative upper bound
and the other for a rough but quick estimate. To get the quantitative
upper bound, one needs an upper bound of the contraction rate of 1-Wasserstein
distance for all pairs of initial values starting from a certain
compact set $\Omega \times \Omega$. This is done by applying extreme value theory. More
precisely, we uniformly sample initial values from $\Omega \times
\Omega$ and compute the contraction rate by using the coupling
method. Then the upper bound of such contraction rate can be obtained
by numerically fitting a generalized Pareto distribution (GPD)
\cite{castillo1997fitting, bali2003generalized}. In practice, one
may want a low cost estimate for the quality of samples.  
Hence we provide a ``rough estimate'' that only
uses the exponential tail of the coupling probability as the rate of
contraction of the generator after a given time $T$. This rough
estimate differs from the true upper bound by an unknown constant, but
it is more efficient and works well empirically.

Our coupling method can be applied to SDEs with degenerate random
terms after suitable modifications. This is done by comparing the overlap of the probability density functions after two
or more steps of the numerical scheme.  Our approach is
demonstrated on a Langevin dynamics example in Section~\ref{sec:langevin_coupling}. It is known from
\cite{eberle2019couplings, bou2018coupling} that a suitable mixture of reflection coupling and
synchronous coupling can be used for Langevin equation.  We find that this
approach can be successfully combined with the ``maximal coupling'' for
the numerical scheme. However, for
SDEs with very degenerate noise, using the coupling method remains to be a
great challenge.

We test our algorithm with a few different examples, from simple to
complicated. The sharpness of our algorithm is checked by using a
``ring density example'' whose invariant probability density function
can be explicitly given. Then we demonstrate the use of coupling
method under degenerate noise by working with a 4D Langevin equation. Next
we show two examples whose numerical invariant probability differs $\hat{\pi}$
significantly from true invariant probability measure $\pi$. One is
an asymmetric double well potential whose transition kernel has a slow rate of
convergence. The other example is the Lorenz 96 model whose
finite time truncation error is very difficult to control due to intensive
chaos. Finally, we study a coupled FizHugh-Nagumo oscillator model
proposed in \cite{chen2019spatial, karimi2010extensive} to
demonstrate that our algorithm works reasonably well in high
dimensional problems.

The organization of this paper is as follows. Section~\ref{sec:probability} serves as the
probability preliminary, in which we review some necessary background
about the coupling method, stochastic differential equations, and
numerical SDE schemes. The main algorithm is developed in Section
3. All numerical examples are demonstrated in Section 4. Section 5 is
the conclusion.

\section{Probability preliminary}
\label{sec:probability}
In this section, we provide some necessary probability preliminaries
for this paper, which are about the coupling method, stochastic
differential equations, numerical stochastic differential equations, and convergence analysis.
\subsection{Coupling}
This subsection provides the definition of coupling of random
variables and Markov processes.

\begin{de}[Coupling of probability measures]
Let $\mathbb{P}$ and $\mathbb{P}'$ be two probability measures on a probability space
$(\Omega,\mathcal{F})$. A probability measure $\gamma$ on $(\Omega \times \Omega,
\mathcal{F} \times \mathcal{F})$ is called a
\textit{coupling} of $\mathbb{P}$ and $\mathbb{P}'$, if two
marginals of $\gamma$ coincide with $\mathbb{P}$ and $\mathbb{P}'$
respectively.   
\end{de} 

The definition of coupling can be extended to any two random variables
that take value in the same state space. 

\begin{de}[Markov Coupling]
A \textit{Markov coupling} of two Markov processes $X_t$ and $Y_t$
with transition kernel $P$ is a Markov process $(\tilde{X}_t, \tilde{Y}_t)$ on the
product state space $V\times V$ such that 
\begin{itemize}
  \item[(i)] The marginal processes $X_t$ and $Y_t$ are Markov
    processes with transition kernel $P$, and
\item[(ii)] If $\tilde{X}_{s} = \tilde{Y}_{s}$, we have $\tilde{X}_{t} = \tilde{Y}_{t}$ for all $t > s$.
\end{itemize}
\end{de}

Markov coupling can be defined in many different ways. For example, let $P$ be
the transition kernel of a Markov chain $X_{t}$ on a countable state space
$V$, the
following transition kernel $Q$ for $(\tilde{X}_t, \tilde{Y}_t)$ on $V\times V$ such that
$$
Q^t((x_1,y_1),(x_2,y_2))=\left\{
\begin{array}{ll}
P^t(x_1,x_2)P^t(y_1,y_2), & \text{if } x_1\neq y_1 \\
P^t(x_1,x_2), & \text{if } x_1=y_1 \text{ and } x_2=y_2\\
0, & \text{if } x_1=y_1 \text{ and } x_2\neq y_2
\end{array}\right.,
$$
is called the {\it independent coupling}. Paths of the
two marginal processes are independent until they first meet. In the
rest of this paper, unless otherwise specified, we only consider
Markov couplings. 

\subsection{Wasserstein distance and total variation distance}
In order to give an estimate for the coupling of Markov processes, we need the following to metrics. 
\begin{de}[Wasserstein distance]
Let $d$ be a metric on the state space $V$. For probability measures $\mu$ and $\nu$ on $V$, the \textit{Wasserstein distance} between $\mu$ and $\nu$ for $d$ is given by
\begin{align*}
d_w(\mu,
  \nu)&=\inf\{\mathbb{E}_{\gamma}[d(x, y)] \,
        :\, \gamma \text{ is a coupling of } \mu \text{ and } \nu.\}\\
&=\inf\Big\{\int d(x,y) \gamma(dx,dy))
 \,: \, \gamma \text{ is a coupling of } \mu \text{ and } \nu.\Big\}.
\end{align*}
\end{de}

For the discussion in our paper, unless otherwise specified, we will use the Wasserstein distance with the distance
\begin{equation}\label{specific_distance}
d(x,y) = \max\{1, \|x - y\| \}, \quad x, y \in \mathbb{R}^{n}.
\end{equation}

\begin{de}[Total variation distance]
Let $\mu$ and $\nu$ be probability measures on $(\Omega,\mathcal{F})$. The \textit{total variation distance} of $\mu$ and $\nu$ is given by
$$d_{\text{TV}}(\mu,\nu) = \|\mu-\nu\|_{\text{TV}} := \sup_{A\in\mathcal{F}}|\mu(A)-\nu(A)|.$$
\end{de}

\subsection{Coupling Lemma}
In this subsection, we provide the coupling inequalities for the approximation of a coupling of Markov processes. 
\begin{de}[Coupling time]
The \textit{coupling time} of two stochastic processes $X_t$ and $Y_t$ is a random variable given by
\begin{equation}
\tau_{\text{c}}=\tau_{\text{c}}(X_t,Y_t):=\inf\{t\geq 0:X_t=Y_t\}.
\end{equation}
\end{de}

\begin{de}[Successful coupling]
A coupling $(\widetilde{X_t},\widetilde{Y}_t)$ of $X_t$ and $Y_t$ is said to be \textit{successful} if
$$\mathbb{P}(\tau_{\text{c}}(\widetilde{X_t},\widetilde{Y}_t)<\infty)=1$$
or equivalently,
$$\lim_{T\rightarrow\infty}\mathbb{P}(\tau_{\text{c}}(\widetilde{X_t},\widetilde{Y}_t)>T)=0.$$
\end{de}

For all Markov couplings, we have the following two coupling inequalities:
\begin{lem}[Coupling inequality w.r.t. the total variation distance]
For the coupling given above, we have
$$\mathbb{P}_{x,y}(\tau_{\text{c}}(\widetilde{X_t},\widetilde{Y}_t)>T)=\mathbb{P}_{x,y}(\widetilde{X}_T\neq\widetilde{Y}_T)\geq d_{\text{TV}}(P^T(x,\cdot),P^T(y,\cdot)).$$
\end{lem}

\begin{proof}
For any $A\in \mathcal{F}$,
\begin{align*}
|P^T(x,A),P^T(y,A)| &= |\mathbb{P}[\widetilde{X}_T\in A]-\mathbb{P}[\widetilde{Y}_T\in A]|\\
&= |\mathbb{P}[\{\widetilde{X}_T\in A\}\cap \{\widetilde{X}_T\neq \widetilde{Y}_T\}] -\mathbb{P}[\{\widetilde{Y}_T\in A\}\cap \{\widetilde{X}_T\neq \widetilde{Y}_T\}]|\\
&\leq \mathbb{P}[\widetilde{X}_T\neq \widetilde{Y}_T],
\end{align*}
where the second equality follows from cancelling the probability
$$\mathbb{P}[\widetilde{X}_T= \widetilde{Y}_T\in A]=\mathbb{P}[\{\widetilde{X}_T\in A\}\cap \{\widetilde{X}_T= \widetilde{Y}_T\}]=\mathbb{P}[\{\widetilde{Y}_T\in A\}\cap \{\widetilde{X}_T= \widetilde{Y}_T\}].$$
By the arbitrariness of $A\in \mathcal{F}$, the lemma is proved.
\end{proof}

\begin{lem}[Coupling inequality w.r.t. the Wasserstein distance]
\label{thm:wasserstein_inequality}
For the coupling given above and the Wasserstein distance induced by the distance given in \eqref{specific_distance}, we have
$$\mathbb{P}_{x,y}(\tau_{\text{c}}(\widetilde{X_t},\widetilde{Y}_t)>T)=\mathbb{P}_{x,y}(\widetilde{X}_T\neq\widetilde{Y}_T)\geq d_w(P^T(x,\cdot),P^T(y,\cdot)).$$
\end{lem}

\begin{proof}
By the definition of the Wasserstein distance,
\begin{align*}
d_w(P^T(x,\cdot),P^T(y,\cdot))&\leq \int d(x,y)\mathbb{P}((\widetilde{X}_T,\widetilde{Y}_T)\in(dx,dy))\\
&=\int_{\{x\neq y\}}d(x,y)\mathbb{P}((\widetilde{X}_T,\widetilde{Y}_T)\in(dx,dy))\\
&\leq \int_{\{x\neq y\}}\mathbb{P}((\widetilde{X}_T,\widetilde{Y}_T)\in(dx,dy))\\
&=\mathbb{P}(\widetilde{X}_T\neq\widetilde{Y}_T),
\end{align*}
where $d(x,y)$ is the specific distance given in \eqref{specific_distance}.
\end{proof}

\subsection{Stochastic differential equations (SDE)}
We consider the following stochastic differential equation (SDE) with initial condition $X(0)=x_0$ that is measurable with respect to $\mathcal{F}_0=\sigma\{B(0)\}$ 
\begin{equation}
\label{SDE}
  \mathrm{d}X_{t} = f(X_{t}) \mathrm{d}t + \sigma(X_{t})
  \mathrm{d}W_{t} \,,
\end{equation}
where $f(X_{t})$ is a continuous vector field in $\mathbb{R}^{n}$,
$\sigma(X_{t})$ is an $n \times m$ matrix-valued function, and
$\mathrm{d}W_{t}$ is the white noise in $\mathbb{R}^{m}$. The
following theorem is well known for the existence and uniqueness of
the solution of equation \eqref{SDE} \cite{mao2006stochastic}.

\begin{thm}
Assume that there are two positive constants $K_1$ and $K_2$ such that the two functions $f$ and $\sigma$ in \eqref{SDE} satisfy
\begin{enumerate}
\item (Lipschitz condition) for all $x,y\in \mathbb{R}^n$ and $t\in [t_0,T]$
\begin{equation}\label{Lipshchitz}
|f(x)-f(y)|^2+|\sigma(x)-\sigma(y)|^2\leq K_1|x - y|^2;
\end{equation}
\item (Linear growth condition) for all $x,y\in \mathbb{R}^n$ and $t\in [t_0,T]$
\begin{equation}\label{Linear_growth}
|f(x)|^2+|\sigma(x)|^2\leq K_2(1+|x|^2).
\end{equation}
\end{enumerate}
Then there exists a unique solution $X(t)$ to equation \eqref{SDE} in $\mathcal{M}^2([t_0,T];R^n)=\{g:g(t) \text{ is } \mathcal{F}_t\text{-adapted and } \mathbb{E}(\int_{t_0}^T |g(t)|^2\,dt)<\infty\}$.
\end{thm}

In addition, we assume that $X_{t}$ admits a unique invariant probability
measure $\pi$. The existence and uniqueness of $\pi$ usually follows
from some drift condition plus some suitable irreducibility conditions
\cite{khasminskii2011stochastic, huang2015steady}. 

\subsection{Numerical SDE}
In this subsection, we talk about the numerical scheme we use for
sampling stochastic differential equations \eqref{SDE}. The
Euler--Maruyama approximation $\bar{X}^h_t$  of the solution $X_t$ of \eqref{SDE} is given by
\begin{equation}
\bar{X}^h_{t}=\hat{X}^h_{t_{k-1}}+f(\hat{X}^h_{t_{k-1}})(t-t_{k-1})+\sigma(\hat{X}^h_{t_{k-1}})(B(t)-B(t_{k-1})),
\end{equation}
where $\hat{X}^h_{0}=x_0$, $t_k=t_0+kh$, $t\in[t_{k-1},t_k]$, and $B(t)-B(t_{k-1})\sim\text{N}(0,t-t_k)$ and $B(t_j)-B(t_{j-1})\sim\text{N}(0,h),j=1,2,\dots,k-1$ are mutually independent Gaussian random vectors. 

We have the following convergence rate for the Euler--Maruyama approximation (see \cite{mao2006stochastic})
\begin{thm}
Assume that the Lipschitz condition \eqref{Lipshchitz} and the linear growth condition \eqref{Linear_growth} hold. Let $\hat{X}_{t}$ be the unique solution of equation \eqref{SDE}, and $\hat{X}^h_{t}$ be the Euler--Maruyama approximation for $t\in[t_0,T]$. Then
\begin{equation}
\mathbb{E}\Big(\sup_{t_0\leq t\leq T}|\hat{X}^h_{t}-X_{t}|^2\Big)\leq Ch,
\end{equation}
where $C$ is a constant depending only on $K_1, K_2, t_0, T$ and $x_0$. 
\end{thm}
Namely, the Euler--Maruyama approximation provides a convergence rate
of order $1/2$. 

A commonly used improvement of the Euler--Maruyama
scheme is called the Milstein scheme, which reads
\begin{align}\label{Milstein}
\hat{X}^h_{t}=\hat{X}^h_{t_{k-1}}+f_k(\hat{X}^h_{t_{k-1}})\Delta t+\sigma(\hat{X}^h_{t_{k-1}})\Delta B+\sigma(\hat{X}^h_{t_{k-1}})IL,
\end{align}
where $\Delta t=t-t_{k-1}$, $\Delta B=B(t)-B(t_{k-1})$ and $I$ is an $m\times m$ matrix with its $(i,j)$-th component being the double It\^{o} integral
$$I_{i,j}=\int_{t_k}^{t}\int_{t_k}^{s_2}\,dB^i(s_1)dB^{j}(s_2),$$
and $L\in\mathbb{R}^m$ is a vector of operators with $i$th component
$$L_i=\sum_{i=1}^{n}\sigma_{i,j}(\hat{X}^h_{t_{k-1}})\frac{\partial}{\partial x_i}.$$

Under suitable assumptions of Lipschitz continuity and linear growth conditions
for some functions of the coefficients $f$ and $\sigma$, the Milstein
scheme \cite{kloeden2013numerical} is an order $1$ strong approximation.

\begin{thm}[\cite{kloeden2013numerical}]
Under suitable assumptions, we have the following estimate for the Milstein approximation $X^h$,
$$\mathbb{E}\big(|\hat{X}^h_{t}-X_{t}|\big)\leq Kh,$$
where $K$ is a constant independent of $h$.
\end{thm}

It is easy to see that when $\sigma(X_{t})$ is a constant matrix, the
Euler--Maruyama scheme and the Milstein scheme coincide. In other
words the Euler--Maruyama scheme for constant $\sigma(X_{t})$ also has
a convergence rate of order $1$. There are also strong approximations
of order $1.5$ or $2$ that are much more complicated to implement. We
refer interested readers to \cite{kloeden2013numerical}. 

\subsection{Extreme value theory}
This subsection introduces some extreme value theory that is relevant
to materials in this paper.

\begin{de}[Generalized Pareto distribution]
A random variable $Y$ is said to follow a \textit{generalized Pareto distribution}, if its cumulative distribution function is given by
$$
F_{\xi,\beta}(x)=\left\{
\begin{array}{ll}
1-\big(1+\frac{\xi x}{\beta}\big)^{-1/\xi}, & \text{if } \xi\neq0,\\
1-\exp(-\frac{x}{\beta}), & \text{if } \xi=0.
\end{array}
\right.
$$
\end{de}

The generalized Pareto distribution is used to model the so-called
peaks over threshold distribution, that is, the part of a random variable over a chosen threshold $u$, or the tail of a distribution. Specifically, for a random variable $X$, consider the random variable $X-u$ conditioning on that the threshold $u$ is exceeded. Its conditional distribution function is called the \textit{conditional excess distribution function} and is denoted by
$$F_u(x) = P(X-u\leq x|X>u) =\frac{P(\{X-u\leq x\}\cap \{X>u\})}{P(X>u)}=\frac{F(x+u)-F(u)}{1-F(u)}.$$
The extreme value theory proves the following theorem.
\begin{thm}[\cite{balkema1974residual, pickands1975statistical}]
For a large class of distributions (e.g., uniform, normal, log-normal, $t$, $F$, gamma, beta distributions), there is a function $\beta(u)$ such that
$$\lim_{u\rightarrow\bar{x}}\sup_{0\leq x<\bar{x}-u}|F_u(x)-F_{\xi,\beta(u)}(x)|=0,$$
where $\bar{x}$ is the rightmost point of the distribution.
\end{thm}

This theorem shows that the conditional distribution of peaks over
threshold can be approximated by the generalized Pareto
distribution. In particular, if the parameter fitting of $(X - u)|_{X
  \geq u}$ gives $\xi < 0$,
we can statistically conclude that the random variable $X$ is has a
bounded distribution with an upper bound $u - \zeta/\xi$. In this
paper, we will employ this method to estimate the upper bound of
quantities of interest.

\section{Description of algorithm}
\subsection{Decomposition of error terms}
Let $X_{t}$ and $\hat{X}_{t}$ be the stochastic process given by
equation \eqref{SDE} and its numerical scheme, respectively. Let $P$
and $\hat{P}$ be two corresponding transition kernels. The time step
size of $\hat{X}_{t}$ is $h$ unless otherwise specified. Hence in
$\hat{X}_{t}$, $t$ only takes values $0, h, 2h, \cdots$. Let $T>0$ be a
fixed constant. Let $\mathrm{d}_{w}$ be the 1-Wasserstein distance of
probability measures induced by distance 
$$
  d(x,y) = \min\{1, \|x - y\| \}, \quad x, y \in \mathbb{R}^{n} 
$$
unless otherwise specified. 

Denote the invariant probability measures of $X_{t}$ and $\hat{X}_{t}$
by $\pi$ and $\hat{\pi}$ respectively. The following decomposition
follows easily by the triangle inequality and the invariance.
(This is motivated by \cite{johndrow2017error}. Similar approaches are
also reported in \cite{rudolf2018perturbation, mitrophanov2005sensitivity}.) 
\begin{equation}
  \label{decomposition}
\mathrm{d}_{w}(\pi, \hat{\pi}) \leq \mathrm{d}_{w}(\pi P^{T}, \pi
\hat{P}^{T}) + \mathrm{d}_{w}( \pi \hat{P}^{T} , \hat{\pi} \hat{P}^{T}
) \,.
\end{equation}
If the transition kernel $\hat{P}^{T}$ has enough contraction such
that 
$$
  \mathrm{d}_{w}( \pi \hat{P}^{T} , \hat{\pi} \hat{P}^{T}) \leq \alpha
  \mathrm{d}_{w}( \pi , \hat{\pi} ) \,,
$$
for some $\alpha < 1$, we have
$$
  \mathrm{d}_{w}(\pi, \hat{\pi}) \leq \frac{\mathrm{d}_{w}(\pi P^{T}, \pi
\hat{P}^{T})}{1 - \alpha} \,.
$$
In other words the distance $\mathrm{d}_{w}(\pi, \hat{\pi})$ can be
estimated by computing the finite time error and the speed of
contraction of $P^{T}$. Theoretically, the finite time error can be
given by the strong approximation of the truncation error of the numerical scheme of equation \eqref{SDE}. The second term comes from the geometric ergodicity of
the Markov process $\hat{X}_{t}$. As discussed in the introduction, except some special cases, 
the rate of geometric ergodicity of $\hat{X}_{t}$ can not be estimated
sharply. As a result one can only have an $\alpha$ that is extremely close to $1$. This makes
decomposition \eqref{decomposition} less interesting in practice. Therefore, we
need to look for suitable numerical estimators of the two terms in
equation \eqref{decomposition}.

\medskip

The other difficulty comes from the fact that both $X_{t}$ and
$\hat{X}_{t}$ are defined on an unbounded domain. However, large
deviations theory guarantees that the mass of both $\pi$ and $\pi^{T}$
should concentrate near the global attractor
\cite{khasminskii2011stochastic, freidlin1998random}. Similar
concentration estimates can be made by many different approaches
\cite{li2016systematic, huang2018concentration}. Therefore, we assume
that there exists a compact set $\Omega$ and a constant $0 < \epsilon \ll
1$, such that 
\begin{equation}
  \label{assumption}
\pi( \Omega^{c}) < \epsilon \, ,  \quad \hat{\pi}( \Omega^{c}) <
\epsilon \, , \quad \pi \hat{P}^{T}( \Omega^{c}) < \epsilon \,. 
\end{equation}
In practice, $\Omega$ can be chosen to be the set that contains all
samples of a very long trajectory of $\hat{X}_{t}$, and $\epsilon$ is
the reciprocal of the length of this trajectory. This $\epsilon$ is
usually significantly smaller than all other error terms.

This allows us to estimate the contraction rate $\alpha$ for initial
values in a compact set. Let $\Gamma$ be
the optimal coupling plan such that
$$
  \mathrm{d}_{w}(\pi, \hat{\pi}) = \int_{\mathbb{R}^{n} \times
    \mathbb{R}^{n}} d(x,y) \Gamma( \mathrm{d}x, \mathrm{d}y) \,.
$$
Let $\hat{P}\circ \hat{P}$ be the transition kernel on $\mathbb{R}^{n}
\times \mathbb{R}^{n}$ that gives a Markov coupling of two
trajectories of $\hat{X}_{t}$. By the assumption of $\Omega$, we have
\begin{align}
  \label{bound}
\mathrm{d}_{w}( \pi P^{T}, \hat{\pi}P^{T}) &\leq \int_{\mathbb{R}^{n}
\times \mathbb{R}^{n}} d(x,y) \Gamma (\hat{P}\circ \hat{P})^{T}(
  \mathrm{d}x, \mathrm{d}y)\\\nonumber
&\leq 2 \epsilon + \int_{\Omega \times \Omega} d(x,y) \Gamma (\hat{P}\circ \hat{P})^{T}(
  \mathrm{d}x, \mathrm{d}y) \\\nonumber
&\leq 2 \epsilon + \alpha_{\Omega} \int_{\Omega \times \Omega} \mathrm{d}(x,y)
  \Gamma( \mathrm{d}x, \mathrm{d}y) \\\nonumber
&\leq 2 \epsilon + \alpha_{\Omega} \mathrm{d}_{w}(\pi, \hat{\pi}) \,,
\end{align}
where $\alpha_{\Omega}$ is the minimum contracting rate of $(\hat{P}\circ \hat{P})^{T}$ on $\Omega
\times \Omega$ such that
$$
  \alpha_{\Omega} = \sup_{(x,y) \in \Omega \times \Omega}
  \frac{\mathrm{d}_{w}(\delta_{x}\hat{P}^{T}, \delta_{y}\hat{P}^{T})}{d(x,y) }\,,
$$
where $\delta_{x} \hat{P}^{T}$ and $\delta_{y} \hat{P}^{T}$ are two
margins of $\delta_{(x,y)}(\hat{P}\circ \hat{P})^{T}$. 

Combine equations \eqref{decomposition} and
\eqref{bound}, we have
\begin{equation}
\label{estimator}
  \mathrm{d}_{w}(\pi, \hat{\pi}) \leq \frac{\mathrm{d}_{w}(\pi P^{T}, \pi
\hat{P}^{T}) + 2 \epsilon}{1 - \alpha_{\Omega}} \,.
\end{equation}

\subsection{Estimator of error terms}
From equation \eqref{estimator}, we need to numerically estimate the
finite time error $\mathrm{d}_{w}(\pi P^{T}, \pi\hat{P}^{T})$ and the contraction rate $\alpha_{\Omega}$. We
propose the following approach to estimate these two quantities.

{\bf Extrapolation for finite time error.} By the definition of
1-Wasserstein distance, we have
$$
  \mathrm{d}_{w}(\pi P^{T}, \pi\hat{P}^{T}) \leq \int_{\mathbb{R}^{n}
    \times \mathbb{R}^{n}} d(x,y) \Gamma( \mathrm{dx}, \mathrm{d}y) \,
$$
for any coupling measure $\Gamma$. A suitable choice of $\Gamma$ that can be sampled easily is
$\pi^{2}(P^{T}\circ \hat{P}^{T})$, where $\pi^{2}$ is the coupling
measure of $\pi$ on the ``diagonal'' of $\mathbb{R}^{n}\times
\mathbb{R}^{n}$ that is supported by the hyperplane
$$
  \{(\mathbf{x}, \mathbf{y}) \in \mathbb{R}^{2n} \,|\, \mathbf{y} =
  \mathbf{x} \} \,.
$$
Theoretically, we can sample an initial value $\pi$, run $X_{t}$ and $\hat{X}_{t}$ up
to time $T$, and calculate $d(X_{T}, \hat{X}_{T})$. However, we do not have exact
expressions for $\pi$ and $X_{t}$. Hence in the estimator, we use $\hat{\pi}$
to replace $\pi$ and use extrapolation to estimate $d(X_{T},
\hat{X}_{T})$.  
The idea is that
$$
  \int_{\mathbb{R}^{n}
    \times \mathbb{R}^{n}} d(x,y) \pi^{2}(P^{T}\circ \hat{P}^{T})( \mathrm{dx}, \mathrm{d}y) 
$$
can be approximated by
$$
  \int_{\mathbb{R}^{n}
    \times \mathbb{R}^{n}} d(x,y) \hat{\pi}^{2}(P^{T}\circ \hat{P}^{T})(
  \mathrm{dx}, \mathrm{d}y)  \,,
$$
as this only induces a higher order error $O(\mathrm{d}^{2}_{w}(\pi,
\hat{\pi}))$. Then the integral above can be estimated by sampling
$\mathrm{d}(X_{T}, \hat{X}_{T})$ such that $X_{0} = \hat{X}_{0} \sim
\pi$. The distance $d(X_{T}, \hat{X}_{T})$ can be obtained by extrapolating
$d(\hat{X}_{T}, \hat{X}^{2h}_{T})$, where $\hat{X}^{2h}_{T}$ is the
random process of the same numerical scheme with the same noise term
but $2h$ time step size. This gives Algorithm \ref{finitetime}.

\begin{algorithm}[h]
\caption{Estimate finite time error}
\label{finitetime}
\begin{algorithmic}
\State {\bf Input:} Initial value $\mathbf{x}_{1}$
\State {\bf Output:} An estimator of $\mathrm{d}_{w}(\pi P^{T}, \pi\hat{P}^{T}) $ 
\For {i = 1 to N}
\State Using the same noise, simulate $\hat{X}_{t}$ and
  $\hat{X}^{2h}_{t}$ with initial value $\mathbf{x}_{i}$ up to $t = T$
\State Let $y_{i} = c d(\hat{X}_{T}, \hat{X}^{2h}_{T})$, where
  $c$ is a constant depending on the order of accuracy
\State Let $\mathbf{x}_{i+1} = \hat{X}_{T}$
\EndFor
\State Return $
  \frac{1}{N}\sum_{i = 1}^{N} y_{i}$
\end{algorithmic}
\end{algorithm}

When $N$ is sufficiently large, $\mathbf{x}_{1},\cdots,
\mathbf{x}_{N}$ in Algorithm \ref{finitetime} are from a long trajectory of the time-$T$ skeleton of
$\hat{X}_{T}$. Hence $\mathbf{x}_{1},\cdots,
\mathbf{x}_{N}$ are approximately sampled from $\hat{\pi}$. The error
term $y_{i} = c d(\hat{X}_{T}, \hat{X}^{2h}_{T})$ for $\hat{X}^{h}_{0} =
\hat{X}^{2h}_{0} = \mathbf{x}_{i}$ estimates $d(X_{T}, \hat{X}_{T})$. Therefore,
$$
  \frac{1}{N}\sum_{i = 1}^{N}y_{i}
$$
estimates the integral 
$$
  \int_{\mathbb{R}^{n}
    \times \mathbb{R}^{n}} d(x,y) \hat{\pi}^{2}(P^{T}\circ
  \hat{P}^{T})( \mathrm{dx}, \mathrm{d}y)  \,,
$$
which is an upper bound of $\mathrm{d}_{w}(\pi P^{T}, \pi\hat{P}^{T})$. The constant $c$ in Algorithm \ref{finitetime} depends on the order of
accuracy. For example, if the numerical scheme is a strong
approximation with order $1$ error $O(h)$,
then $c = 1$ because $d(X_{T}, \hat{X}_{T}) \approx d(\hat{X}_{T},
\hat{X}^{2h}_{T})$ when $h \ll 1$. 

One advantage of Algorithm \ref{finitetime} is that it can run
together with the Monte Carlo sampler. The trajectory of
$\hat{X}^{2h}_{t}$ can not be recycled. But the trajectory of
$\hat{X}_{t}$ can be used to estimate either the invariant density or
the expectation of an observable. 

{\bf Coupling for contraction rate.} The idea of estimating 
$\alpha_{\Omega}$ is to use coupling. We can construct a Markov
process $\hat{Z}_{t} = (\hat{X}^{(1)}_{t}, \hat{X}^{(2)}_{t})$ such
that $\hat{Z}_{t}$ is a Markov coupling of $\hat{X}^{(1)}_{t}$ and
$\hat{X}^{(2)}_{t}$. Then as introduced in Section 2, the first passage time to the ``diagonal'' hyperplane $\{(\mathbf{x}, \mathbf{y}) \in \mathbb{R}^{2n} \,|\, \mathbf{y} =
  \mathbf{x} \} $ is the {\bf coupling time}, which is denoted
  by $\tau_{c}$. It then follows from Lemma \ref{thm:wasserstein_inequality} that 
$$
  \mathrm{d}_{w}(\delta_{x} \hat{P}^{T}, \delta_{y} \hat{P}^{T}) \leq
  \mathbb{P}_{x,y}[\tau_{c} > T] \,.
$$ 

Then we can use extreme value theory to estimate
$\alpha_{\Omega}$. The idea is to uniformly sample initial values $(x,y)$
from $\Omega \times \Omega$, and define $\beta(x,y) :=
\mathbb{P}_{x,y}[\tau_{c}  > T]/d(x,y)$. Then $\beta$ is actually a random variable whose sample
can be easily computed. We use extreme value theory to estimate an
upper bound for $\beta$, and denote it by $\alpha_{\Omega}$. See
Algorithm \ref{coupling} for details. The threshold $V$ in
Algorithm \ref{coupling} is usually
chosen such that approximately $5 \%$ samples are greater than this
threshold. 

\begin{algorithm}[h]
\caption{Estimate contraction rate}
\label{coupling}
\begin{algorithmic}
\State {\bf Input:} A compact set $\Omega$
\State {\bf Output:} An estimator of the minimal contraction rate $\alpha_{\Omega}$.
\For {i = 1 to N}
\State Sample pairs $(x_{i}, y_{i})$ uniformly from $\Omega \times
\Omega$
\State Set $K_{i} = 0$, $r_{i} = 0$
\For {j = 1 to M}
\State Run $\hat{Z}_{t}$ with initial value $(x_{i}, y_{i})$ until
$\min\{\tau_{c}, T\}$
\If {$\tau_{c} \leq T$}
\State $K_{i} \gets K_{i} + 1$
\EndIf
\EndFor
\State $r_{i} \gets K_{i}/(d(x_i,y_i)M)$
\EndFor
\If {$\max\{ r_{i} \} \geq 1$}
\State The estimator fails. Choose better coupling algorithm or larger
$T$
\Else
\State Let $v_{i} = 1/(1 - r_{i})$ for all $1 \leq i \leq N$
\State Choose a threshold $V$
\State Use generalized Pareto distribution to fit $\{ v_{i} - V \,|\,
v_{i} \geq V \}$ to get two parameters $(\zeta, \xi)$.
\If {$\xi \geq 0$}
\State The estimator fails. Choose better coupling algorithm or larger $T$
\Else
\State Let $v_{max} = V - \zeta/\xi$.
\EndIf
\EndIf
\State Return $1 - 1/v_{max}$ as the estimator of $\alpha_{\Omega}$.
\end{algorithmic}
\end{algorithm}

If running successfully, Algorithms \ref{finitetime} and
\ref{coupling} give us an upper bound of $\mathrm{d}_{w}(\pi,
\hat{\pi})$ according to equation \eqref{estimator}, which can be used
to check the quality of samples.

\medskip

{\bf Construction of $\hat{Z}_{t}$. } It remains to construct a
coupling scheme that is suitable for the numerical trajectory $\hat{
X}_{t}$. In this paper we use the follow two types of couplings. 

Denote two margins of $\hat{Z}_{t}$ by $\hat{X}^{(1)}_{t}$ and
$\hat{X}^{(2)}_{t}$ respectively. The {\it Reflection coupling} means
two noise terms of $\hat{X}^{(1)}_{t}$ and
$\hat{X}^{(2)}_{t}$ in an update are always symmetric. For the case
of Euler-Maruyama scheme and constant coefficient matrix 
$\sigma(X_{t}) = \sigma$ as an example. A reflection coupling gives the following update
from $t$ to $t+h$:
\begin{align}
  \label{reflection}
&\hat{X}^{(1)}_{t+h} = \hat{X}^{(1)}_{t} + f( \hat{X}^{(1)}_{t})
  \mathrm{d}t + \sigma \sqrt{h} N_{t} \\\nonumber
&\hat{X}^{(2)}_{t+h} = \hat{X}^{(2)}_{t} + f( \hat{X}^{(2)}_{t})
  \mathrm{d}t + \sigma \sqrt{h} (I - 2 e_{t}e_{t}^{T})N_{t} \,, 
\end{align}
where $N_{t}$ is a normal random variable with mean $0$ and variance
$h$, and
$$
  e_{t} = \frac{1}{\| \sigma^{-1}(\hat{X}^{(1)}_{t}  - \hat{X}^{(2)}_{t} )\|}\sigma^{-1}(\hat{X}^{(1)}_{t}  - \hat{X}^{(2)}_{t} )
$$
is a unit vector. It is known that reflection coupling is the
optimal coupling for Brownian motion in $\mathbb{R}^{n}$
\cite{lindvall1986coupling, hsu2013maximal}. Empirically it 
gives fast coupling rates for many stochastic
differential equations with non-degenerate noise. 

The {\it maximal coupling} looks for the maximal coupling probability
for the next step (or next several steps) of the numerical
scheme. Assume $\hat{X}^{(1)}_{t}$ and $\hat{X}^{(2)}_{t}$ are both
known. Then it is easy to explicitly calculate the probability density
function of $\hat{X}^{(1)}_{t+h}$ and $\hat{X}^{(2)}_{t+h}$, denoted
by $p^{(1)}(x)$ and $p^{(2)}(x)$ respectively. The update of
$\hat{X}^{(1)}_{t+h}$ and $\hat{X}^{(2)}_{t+h}$ is described in
Algorithm \ref{max}. This update maximizes the probability of coupling
at the next step. This is similar to the maximal coupling method
implemented in \cite{jacob2017unbiased, johnson1998coupling}. 

\begin{algorithm}[h]
\caption{Maximal coupling}
\label{max}
\begin{algorithmic}
\State {\bf Input:} $\hat{X}^{(1)}_{t}$ and $\hat{X}^{(2)}_{t}$
\State {\bf Output:} $\hat{X}^{(1)}_{t+h}$, $\hat{X}^{(2)}_{t+h}$, and
$\tau$ if coupled.
\State Sample $\hat{X}^{(1)}_{t+h}$ and $\hat{X}^{(2)}_{t+h}$
according to the numerical scheme using independent noise
\State Compute probability density functions $p^{(1)}(x)$ and
$p^{(2)}(x)$.
\State Let
$$
  r = \frac{\min\{ p^{(1)}(\hat{X}^{(1)}_{t+h}),
    p^{(2)}(\hat{X}^{(1)}_{t+h})\}}{\max\{
    p^{(1)}(\hat{X}^{(1)}_{t+h}), p^{(2)}(\hat{X}^{(1)}_{t+h})\}}
  \times \frac{\min\{ p^{(1)}(\hat{X}^{(2)}_{t+h}),
    p^{(2)}(\hat{X}^{(2)}_{t+h})\}}{\max\{
    p^{(1)}(\hat{X}^{(2)}_{t+h}), p^{(2)}(\hat{X}^{(2)}_{t+h})\}}
$$
\State Draw $u$ from uniform 0-1 distribution.
\If {u < r}
\State $\hat{X}^{(1)}_{t+h} = \hat{X}^{(2)}_{t+h}$, $\tau = t+h$
\Else
\State Use  $\hat{X}^{(1)}_{t+h}$ and $\hat{X}^{(2)}_{t+h}$ sampled
before
\State $\tau$ is undetermined. 
\EndIf
\end{algorithmic}
\end{algorithm}

\medskip

In practice, we use reflection coupling when $\hat{X}^{(1)}_{t}$ and
$\hat{X}^{(2)}_{t}$ are far away from each other,
and maximal coupling when they are sufficiently close. This method
fits discrete-time numerical schemes, as using reflection coupling
only will easily let two processes miss each other. In our simulation
code, the threshold of changing coupling method is $2\sqrt{h} \|
\sigma\|$. When the distance between $\hat{X}^{(1)}_{t}$ and
$\hat{X}^{(2)}_{t}$ is smaller than this, we use maximal coupling. 

\subsection{A fast estimator}
In practice, Algorithm \ref{finitetime} can be done together with the
Monte Carlo sampler to compute either an observable or
the invariant probability density function. The extra cost comes from
simulating trajectories of $\hat{X}^{2h}_{t}$, which takes $50 \%$
time of running trajectories of $\hat{X}_{t}$. The main overhead of
the above mentioned methods is algorithm \ref{coupling}. Because we
want a quantitative upper bound of $\mathrm{d}_{w}(\pi, \hat{\pi})$,
the contraction rate of $\hat{P}^{T}$ in 1-Wasserstein space with
respect to all pairs of initial points in $\Omega \times \Omega$ must
be estimated. In practice, this takes a long time because one needs
to run many ($100 - 1000$) independent trajectories from each initial
point to estimate the coupling probability. 

In practice, if one only needs a rough estimate about the sample
quality instead of a definite upper bound of the 1-Wasserstein
distance, Algorithm \ref{coupling} can be done in a much easier
way by estimating the exponential {\it rate} of convergence of
$\hat{P}^{t}$. It is usually safe to assume that the rate of
exponential contraction is same for all ``reasonable'' initial
distributions. In addition, the contraction rate is bounded from below
by the exponential tail of the coupling time distribution
$\mathbb{P}[\tau_{c} > t]$. Therefore, we only need to sample some initial points
uniformly distributed in $\Omega \times \Omega$ and estimate the exponential tail of the
coupling time distribution. This gives an estimate
$$
  \lim_{t \rightarrow \infty} \frac{1}{t}\log (\mathbb{P}_{u}[\tau_{c} >
t])  = - \gamma \,,
$$
where $u$ denotes the uniform distribution on $\Omega \times \Omega$,
and $\gamma$ can be obtained by a linear fit of $\mathbb{P}_{u}[\tau_{c}
> t]$ versus $t$ in a log-linear plot. Then we have a rough estimate of
$\mathrm{d}_{w}(\pi, \hat{\pi})$ given by 
\begin{equation}
  \label{rough}
\mathrm{d}_{w}(\pi, \hat{\pi}) \approx \frac{ \mathrm{d}_{w}(\pi
  P^{T}, \pi \hat{P}^{T}) + 2 \epsilon}{1 - e^{-\gamma T}} \,,
\end{equation}
where $\mathrm{d}_{w}(\pi P^{T}, \pi \hat{P}^{T})$ is estimated by
Algorithm \ref{finitetime}. 

Equation \eqref{rough} usually differs from the output of Algorithm
\ref{finitetime}, Algorithm \ref{coupling}, and equation
\eqref{estimator} by an unknown multiplicative constant. However, in practice, this is
usually sufficient for us to predict the quality of Monte Carlo
sampler with relatively low computational cost. In numerical examples we
will show that, it is usually sufficient to estimate the exponential
tail of $\mathbb{P}_{u}[\tau_{c} > t]$ by running $10^{4} - 10^{5}$
examples.

\section{Numerical examples}

\subsection{Ring density}
The first example is the ``ring density'' example that has a known
invariant probability measure. Consider the following stochastic differential equation:
\begin{equation}\label{ring}
\left\{
\begin{array}{l}
dx=\big(-4x(x^2 + y^2 - 1) + y\big)\,dt + \sigma\,dW_t^1\\
dy=\big(-4y(x^2 + y^2 - 1) - x\big)\,dt + \sigma\,dW_t^2
\end{array}
\right.,
\end{equation}
where $W_t^1$ and $W_t^2$ are independent Wiener processes, and
$\sigma$ is the strength of the noise. The drift part
of equation \eqref{ring} is a gradient flow of potential function
$V(x,y)=(x^2+y^2-1)^2$ plus an orthogonal rotation term that does not change the invariant probability density
function. Hence the invariant probability measure
of \eqref{ring} has a probability density function 
$$u(x,y)=\frac{1}{K}e^{-2V(x,y)/\sigma^2},$$
where $K=\pi\int_{-1}^{\infty}e^{-2t^2/\sigma^2}\,dt$ is a
normalizer. We will compare the invariant probability measure of the
Euler-Maruyama scheme and that of equation \eqref{ring}.

In our simulation, we choose $\sigma = 0.5$ and $T = 10$. The first simulation runs
$8$ independent long trajectories up to time $1.25 \times 10^{6}$. Hence
Algorithm \ref{finitetime} compares the distance between $\hat{X}_{T}$
and $\hat{X}_{T}^{2h}$ for $10^{7}$ samples. Constant $c$ equals $1$ here
because the Euler-Maruyama scheme is a strong approximation with
accuracy $O(h)$ when $\sigma$ is a constant. Algorithm
\ref{finitetime} gives an upper bound
$$
  \mathrm{d}_{w}( \pi P^{T}, \pi \hat{P}^{T}) \leq 0.00141635 \,.
$$
In addition, all eight trajectories are contained in the box $[-2,
2]^{2}$. Hence we choose $\Omega = [-2, 2]^{2}$ and $\epsilon =
10^{-7}$.

\medskip

Then we run Algorithm \ref{coupling} to get coupling probabilities up
to $T = 10$. The number of initial values $(x_{i}, y_{i})$ is
$20000$. Then we run $1000$
pairs of trajectories from each initial point to estimate the coupling
probability. The probability that coupling has not happened before $T$
is then divided by $d(x_{i}, y_{i})$, which estimates an upper bound
of the contraction rate
$$
 r_{i} = \frac{ \mathbb{P}_{x_{i}, y_{i}}[ \tau_{c} > T]}{ d(x_{i},
    y_{i})} \geq \frac{\mathrm{d}_{w}( \delta_{x_{i}}\hat{P}^{T},
    \delta_{y_{i}}\hat{P}^{T})}{d(x_{i}, y_{i})} \,.
$$
Then we use generalized Pareto distribution (GPD) to fit $\{1/(1 -
r_{i})\}_{i = 1}^{20000}$. The threshold $V$ is chosen to be
$1.48$. The fitting algorithm gives parameters $\xi = -0.1822$ and
$\zeta = 0.0326$. This gives $\alpha_{\Omega} = 0.3972$. A comparison
of cumulative distribution functions of empirical data and that of the
GPD fitting is demonstrated in Figure \ref{GPD}.

\begin{figure}[htbp]
\centerline{\includegraphics[width = 0.7\linewidth]{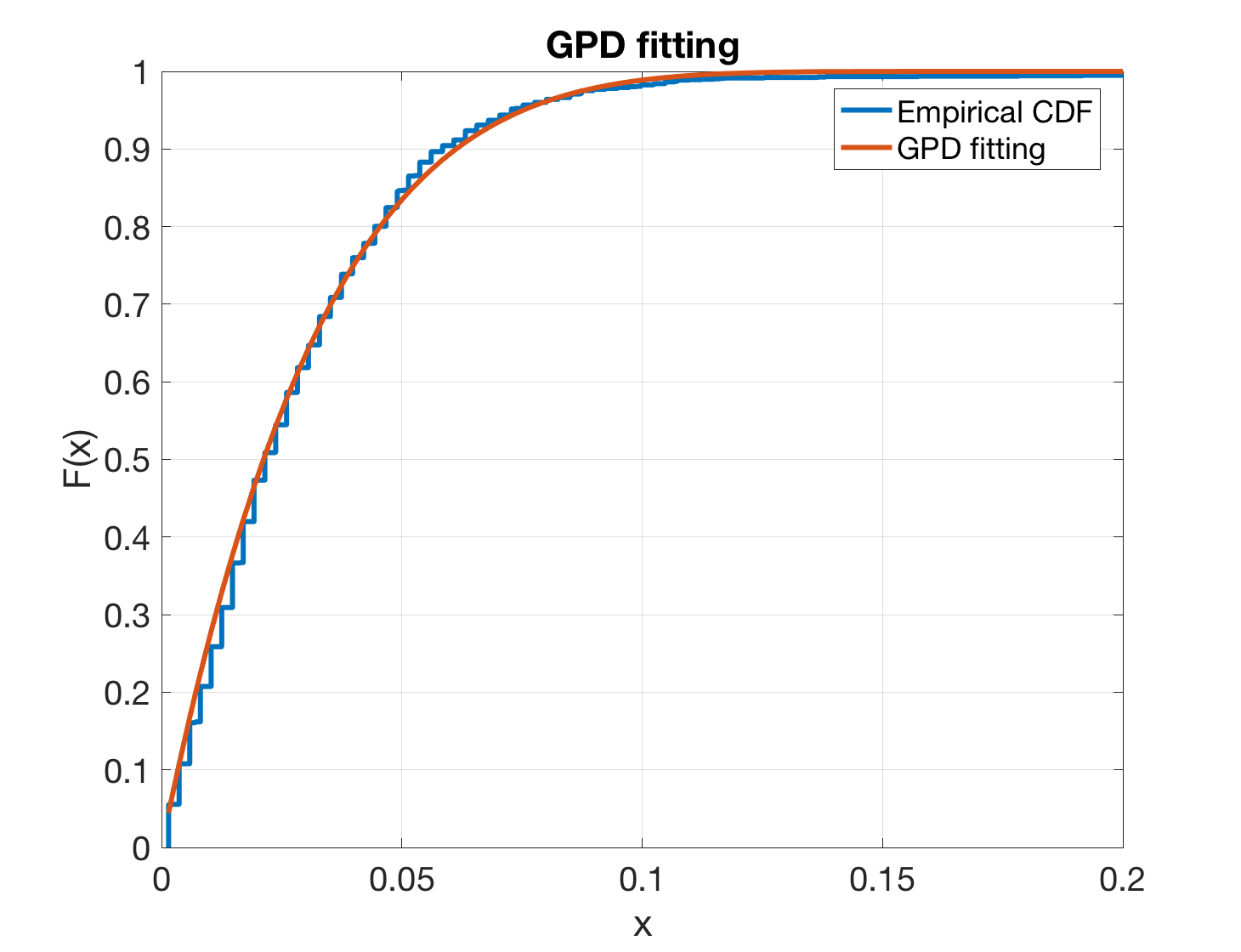}}
\caption{A comparison of cumulative distribution functions of
  empirical data $\{ v_{i} \,|\, v_{i} > V \}$ and that of the
  generalized Pareto distribution.}
\label{GPD}
\end{figure}

\medskip

Combining all estimates above, we obtain a bound
\begin{equation}
\label{ringbound}
  \mathrm{d}_{w}(\pi, \hat{\pi}) \leq 0.002350 \,.
\end{equation}

Since the invariant probability measure of equation \ref{ring} is
known, we can check the sharpness of the bound given in equation \eqref{ringbound}. The
approach we take is Monte Carlo simulation with extrapolation to
infinite sample size. On a $256\times 256$ grid, we use $8$ long
trajectories to estimate the invariant probability density function of
\ref{ring}. The sample sizes of these trajectories are $8\times
10^{8}, 1.6 \times 10^{9}, \cdots, 6.4 \times 10^{9}$. Then we compute
the total variation distance between 
$$u(x,y)=\frac{1}{K}e^{-2V(x,y)/\sigma^2}$$
and the empirical probability density function at those grid
points. The error is linearly dependent on the $-1/2$ power of the sample
size. Linear extrapolation shows that the total variation distance at
the infinite sample limit is $\approx 0.001534$. The linear extrapolation is
demonstrated in Figure~\ref{Extra}. Since $\mathrm{d}_{w}$ is smaller
than the total variation distance, the 1-Wasserstein distance $\mathrm{d}_{w}(\pi,
\hat{\pi})$ should be no greater than $0.001534$. Therefore, our
estimation given in equation \eqref{ringbound} is larger than the true
distance between $\pi$ and $\hat{\pi}$, but is reasonably sharp. 

\begin{figure}[htbp]
\centerline{\includegraphics[width = 0.7\linewidth]{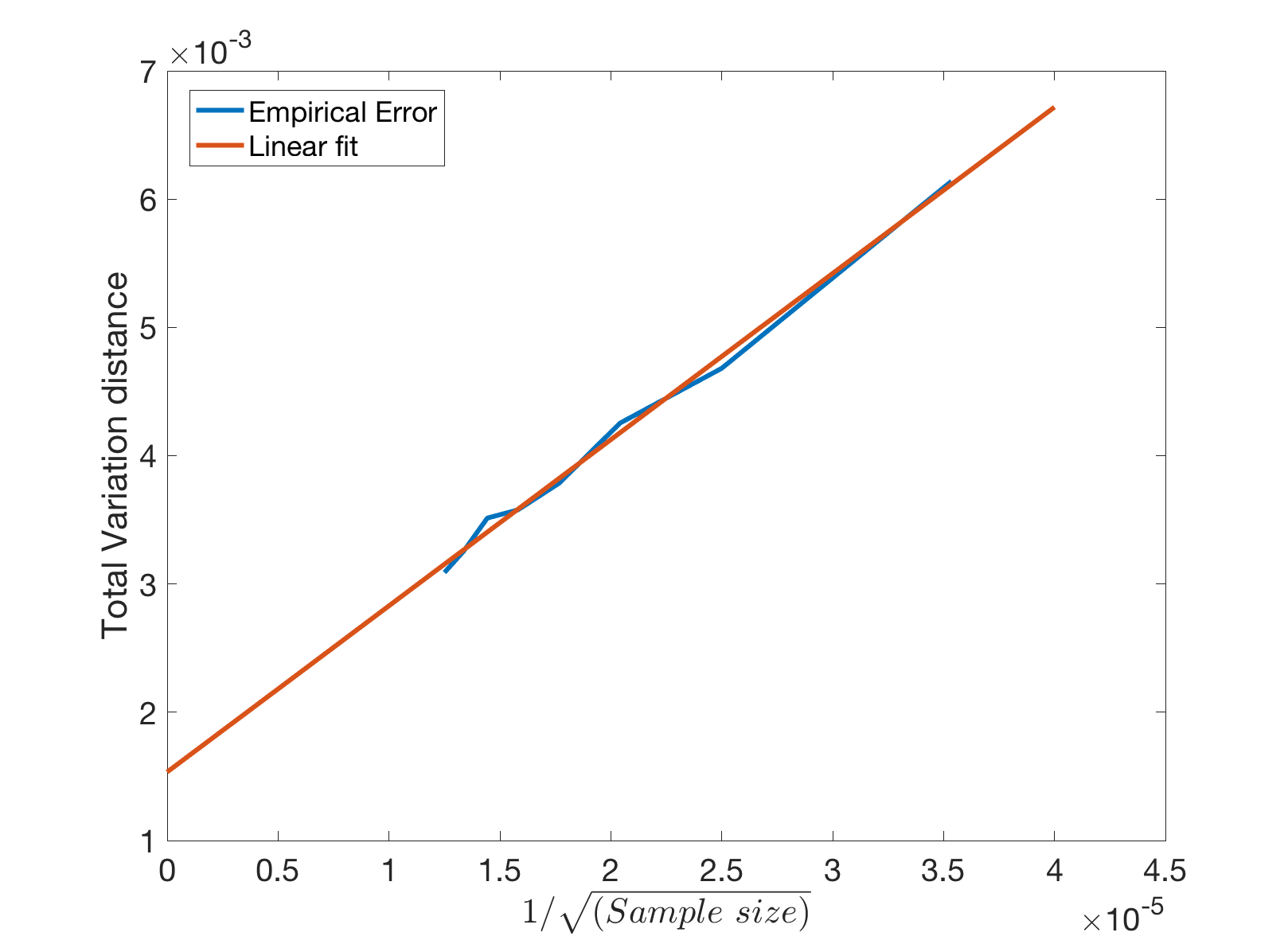}}
\caption{Linear extrapolation for the total variance distance $\|\pi -
  \hat{\pi} \|_{TV}$ at the infinite sample limit.}
\label{Extra} 
\end{figure}

It remains to comment on the fast estimator mentioned in Section
3.3. In Figure \ref{ringtail} we draw the exponential tail of
$\mathbb{P}[\tau_{c} > t]$ and its linear fit. The slope of the
exponential tail is $\gamma = -0.1378$. When $T = 10$, we have
$e^{-\gamma T} = 0.2521$. Equation \eqref{rough} then gives an
estimate
$$
  \mathrm{d}_{w}(\pi, \hat{\pi}) \approx 0.00189377 \,,
$$
which is actually closer to the total variation distance that we have
measured numerically. 

\begin{figure}[htbp]
\centerline{\includegraphics[width = 0.7\linewidth]{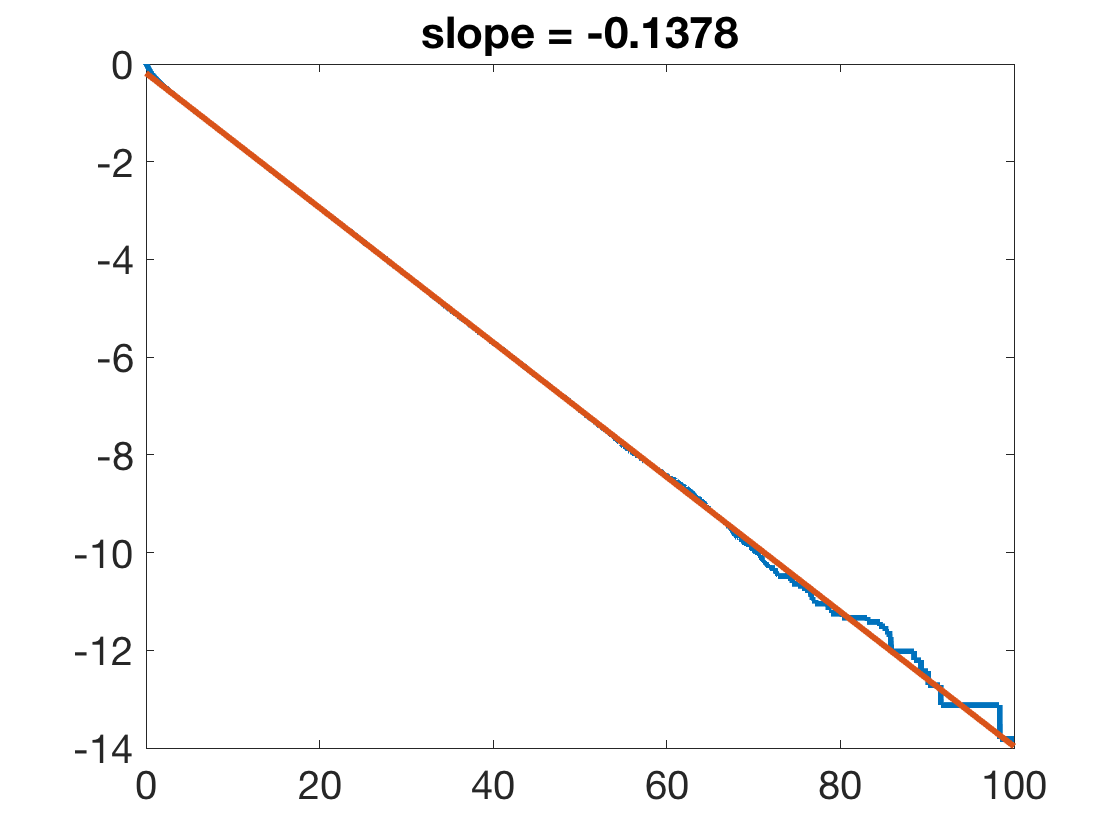}}
\caption{Exponential tail of $\mathbb{P}[\tau_{c} > t]$ versus $t$
  when initial values are uniformly sampled in $\Omega \times \Omega$.}
\label{ringtail} 
\end{figure}

\subsection{Double well potential}
The second example we study is a gradient flow with respect to an
asymmetric double well potential. Let
$$
  V(x) = \left \{
\begin{array}{ll}
6 x^{2} - 60&\mbox{ if } x \geq 4\\
\frac{1}{4}x^{4} - 2x^{2} + 4 &\mbox{ if } 0 \leq x < 4\\
\frac{1}{4}r^{4}x^{4} - 2r^{2}x^{2} + 4 &\mbox{ if } -4/r \leq x < 0\\
6r^{2}x^{2} - 60 & \mbox{ if } x < - 4/r 
\end{array}
\right .
$$
If $r \neq 1$, $V$ is an asymmetric double well potential
function. Note that we make $V(x)$ a quadratic function when $x \geq
4$ or $x < -4/r$, because the original quartic function has very large
derivatives when $|x|$ is large, which has some undesired numerical
artifacts. 

Now consider the gradient flow of $V(x)$ with additive random perturbation
$$
  X_{t} = - V'(x) + \sigma\mathrm{d}W_{t} \,.
$$
It is easy to see that $X_{t}$ admits an invariant probability measure $\pi$
with probability density function
$$
  u(x) = \frac{1}{K}e^{2V(x)/\sigma^{2}} \,,
$$
where $K$ is a normalizer. In this example, we choose $\Omega = [-2,
4]$, as $u(x)$ is extremely small when $x < -2$ or $x > 4$.  

Because of the double well potential, trajectories from two local
minima need a long time to meet with one another. Hence the speed of
convergence of the law of $X_{t}$ to $\pi$ is slow. Much longer times
are needed so that trajectories can couple in Algorithm
\ref{coupling}. In addition, we changed the underlying distance from
$|x-y|$ to $|x-y|^{0.45}$, because if two initial values are very close
to each other, with some small probability one trajectory can run into a
different local minimum and takes a very long time to return. As a result,
for reasonably large $T$, if the underlying distance
$|x-y|$ is used, $\hat{P}^{T}$ does not contract in 1-Wasserstein metric space when two
initial points are very close to each other.

Model parameters are chosen to be $r = 5$, $\sigma =
1.2$, and $T = 50$. The numerical trajectory $\hat{X}_{t}$ is
obtained by running Euler-Maruyama scheme with $h = 0.0025$. We first run Algorithm 1 with $8$ independent long
trajectories to compare the distance between $\hat{X}_{T}$
and $\hat{X}_{2h}$. The length of each trajectory is $5 \times
10^{6}$. The constant $c$ is still equal to $1$ because the accuracy of
Euler-Maruyama scheme is $O(h)$ when $\sigma$ is a constant. Algorithm
1 gives an upper bound
$$
  \mathrm{d}_{w}(\pi P^{T}, \pi \hat{P}^{T}) \leq 0.167345 \,.
$$
The upper bound is quite large due to large second order derivatives
of $V(x)$ and large time span $T$.

Then we run Algorithm 2 to get the contraction rate of
$\hat{P}^{T}$ for $T = 50$. The number of initial values $(x_{i}, y_{i})$ is
$20000$. We run $1000$ pairs of trajectories from each initial points
to get $\mathbb{P}[\tau_{c} > T]$. This gives $20000$ numbers
$$
  r_{i} = \frac{ \mathbb{P}_{x_{i}, y_{i}}[ \tau_{c} > T]}{ d(x_{i},
    y_{i})} \geq \frac{\mathrm{d}_{w}( \delta_{x_{i}}\hat{P}^{T},
    \delta_{y_{i}}\hat{P}^{T})}{d(x_{i}, y_{i})} \,.
$$
Then we use generalized Pareto distribution (GPD) to fit $\{1/(1 -
r_{i})\}_{i = 1}^{20000}$. Similar to in the previous example, GPD
parameter fitting
gives $\alpha_{\Omega} = 0.2019$. See Figure \ref{EVT1D} Left for the
fitting result. In addition, because this is an 1D
problem, we can plot the contraction rate $r_{i}$ for each pair of
$(x_{i}, y_{i})$ on a grid that covers $\Omega \times \Omega$. See
Figure \ref{EVT1D} Right for a heat map of contraction rates from each pairs
of initial points. From Figure \ref{EVT1D} Right we can see that the
high value of $\mathbb{P}_{x_{i}, y_{i}}[ \tau_{c} > T]$ is reached
when one of the pairs $(x,y)$ falls into the left part of the domain, where
$V(x)$ has large derivatives. In addition, when $(x, y)$ becomes
even closer to line $\{x - y = 0\}$, $\mathbb{P}_{x_{i}, y_{i}}[ \tau_{c} > T]$
drops dramatically to $0$. This further confirms that $\hat{P}^{T}$ is
contracting in 1-Wasserstein metric space for $T = 50$. Finally,
we provide an exponential tail of $\mathbb{P}[ \tau_{c} > t]$ for
$\hat{X}_{t}$, demonstrated in Figure \ref{dbwelltail}. The exponential tail is
$\gamma = -0.09078$. Hence $e^{-\gamma T}$ gives $0.1068$, which is
smaller than $\alpha_{\Omega}$ obtained above. 

\begin{figure}[htbp]
\centerline{\includegraphics[width = \linewidth]{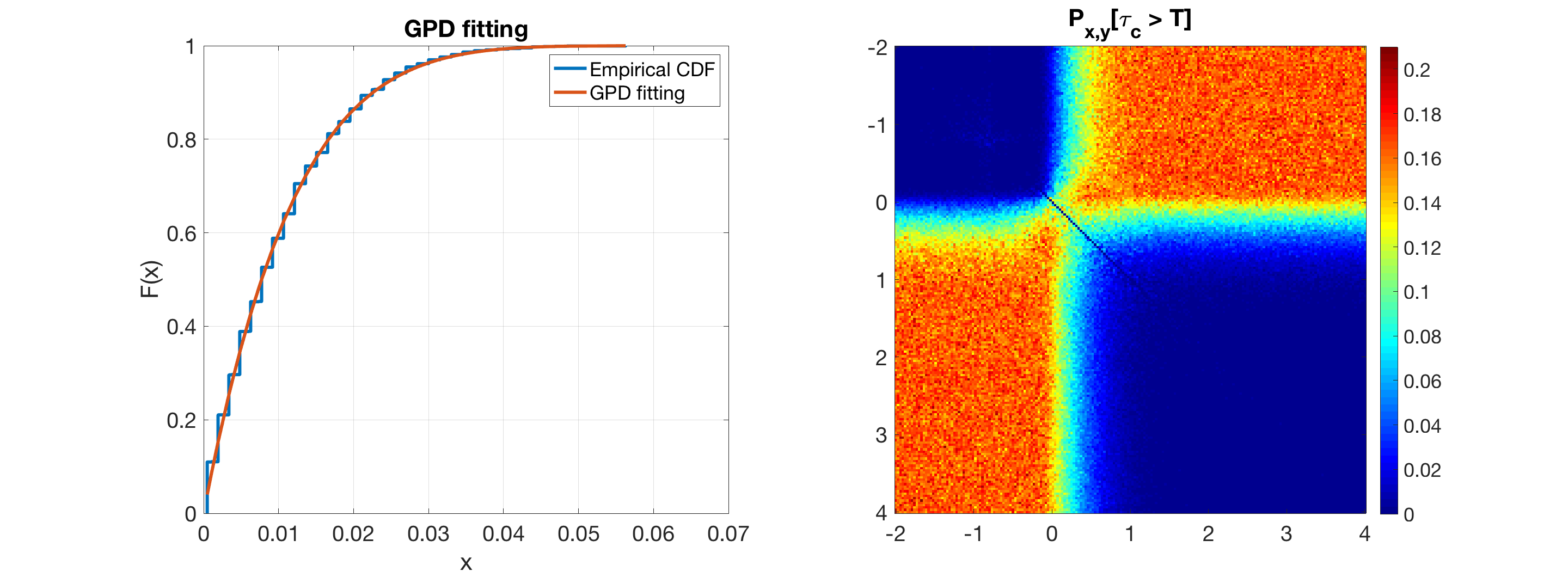}}
\caption{ Left: Fitting generalized Paredo distribution with $v_{i} =
  1/(1 - r_{i})$. The fitting result is compared with the empirical
  cumulative distribution function. Right: Heat map of contraction
  rate $r_{i}$ for initial pairs of points on a grid that covers
  $\Omega \times \Omega$. }
\label{EVT1D}
\end{figure}

\begin{figure}[htbp]
\centerline{\includegraphics[width = 0.7\linewidth]{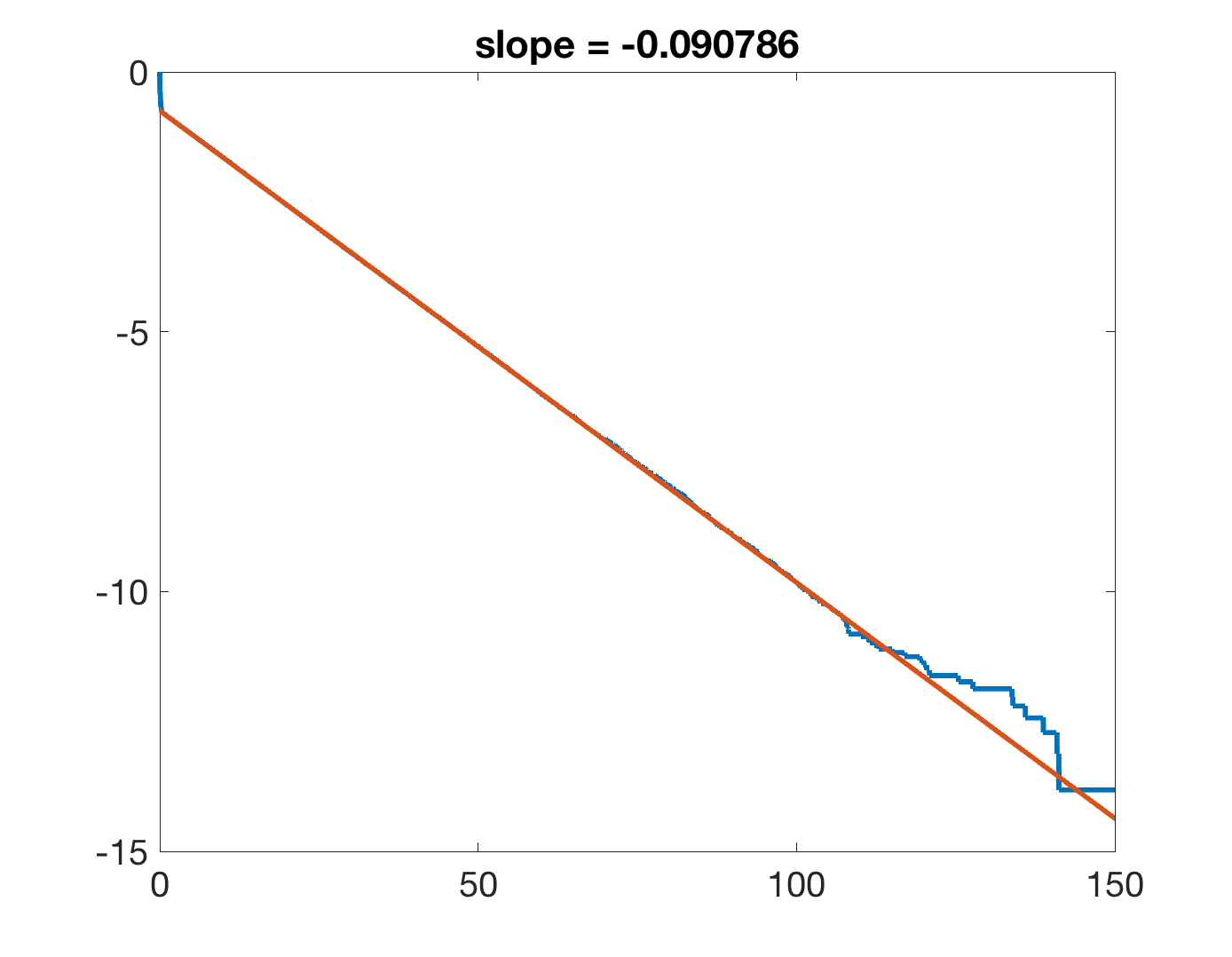}}
\caption{Exponential tail of $\mathbb{P}[\tau_{c} > t]$ versus $t$
  when initial values are uniformly sampled in $\Omega \times \Omega$.}
\label{dbwelltail} 
\end{figure}

Combining all estimates above, we have an upper bound
\begin{equation}
  \label{1Dbound}
\mathrm{d}_{w}(\pi, \hat{\pi}) \leq 0.2097 \,.
\end{equation}
If equation \eqref{rough} is used instead, we have a rough estimate
$$
  \mathrm{d}_{w}(\pi, \hat{\pi}) \approx 0.187354 \,.
$$

Both results imply that two invariant probability measures may be very
different from each other. This can be confirmed by using Monte Carlo
simulation to compute the invariant probability measure of
$\hat{X}_{t}$. We run $8$ independent long trajectories of
$\hat{X}_{t}$ up to $5 \times 10^{6}$ to compute its invariant
probability density function. The result is compared with $u(x)$, the
invariant probability density function of $X_{t}$, in Figure
\ref{compare1D}. We can see visible difference between these two
probability density functions. The total variation difference between them is
$0.05906$. This is smaller than the bound predicted by equation
\eqref{1Dbound}, partially because we have to use the distance induced
by $|x-y|^{0.45}$ to make $r_{i}$ uniformly bounded from above. But our
calculation still predicts an unusually large difference between two invariant
probability measures.

\begin{figure}[htbp]
\centerline{\includegraphics[width= \linewidth]{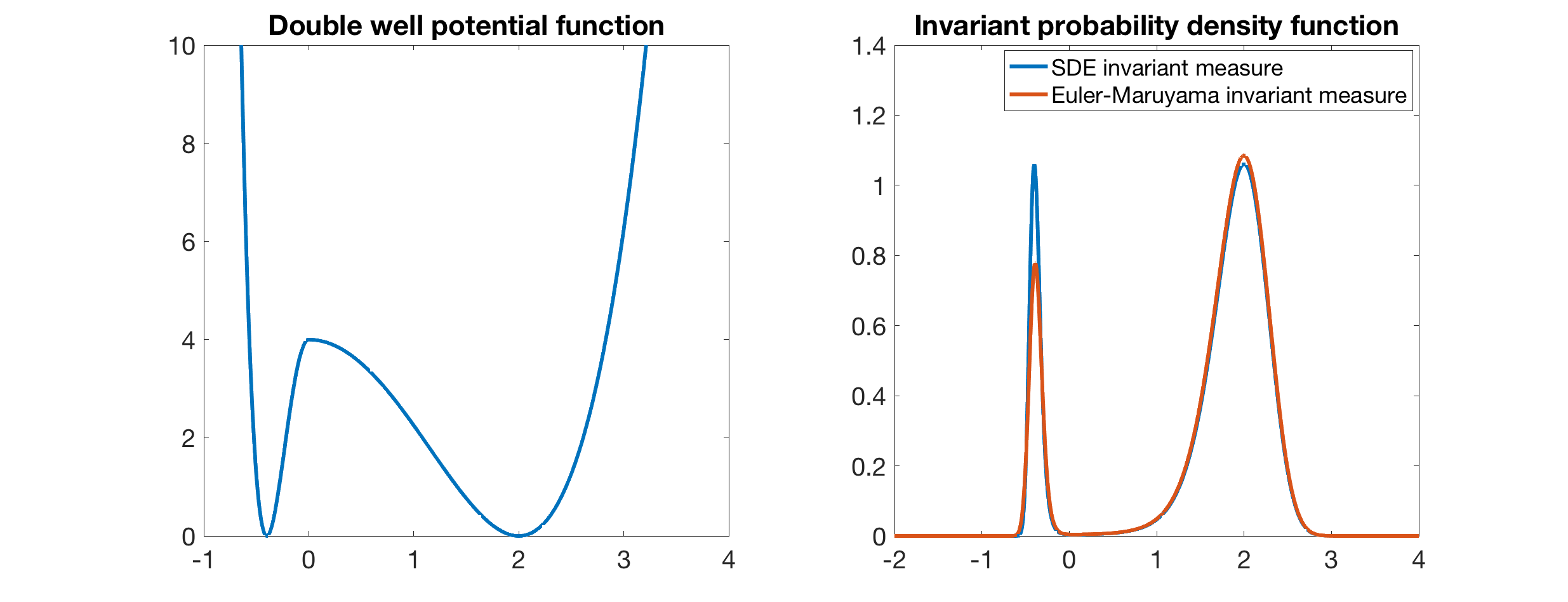}}
\caption{Left: potential function $V(x)$. Right: Comparison of
  invariant probability density functions of $X_{t}$ and $\hat{X}_{t}$.}
\label{compare1D}
\end{figure}

We still owe readers a heuristic explanation of the phenomenon seen in
Figure \ref{compare1D} Right. The probability density of the invariant
probability measure of $\hat{X}_{t}$ is much lower than that of
$X_{t}$ around the local minimum $x = -0.4$ because the potential
function is asymmetric. As a result, when a trajectory of
$\hat{X}_{t}$ moves from $x = -0.4$ to $x = 0$, the Euler-Maruyama
scheme tend to underestimate $-V'(x)$, which increases dramatically
near $x = -0.4$. The effect of such underestimation is much weaker
near the other local minimum $x = 2$, where the value of $|V'(x)|$ is
significantly smaller. As a result, it is easier for the trajectory of
$\hat{X}_{t}$ to pass the separatrix $x = 0$ from left to right than
from right to left. This causes the unbalanced invariant probability
density function as seen in Figure \ref{compare1D} Right.

\subsection{Degenerate diffusion}
\label{sec:langevin_coupling}

Langevin dynamics have noise terms that only appear directly in the velocity
equation and not on the position, leading to a Fokker-Planck equation with
a degenerate, hypoelliptic diffusion.  We consider a potential energy similar
to the ring density equation, with SDE:
\begin{equation}\label{lang_ring}
\begin{split}
d X &= V \, dt \\
d V &= - \nabla U(X) \, dt- \gamma V \, dt + \sigma \, dW \\
\end{split}
\end{equation}
where
\begin{equation}
U(X) = ( X_1^2 + X_2^2 - 1)^2.
\end{equation}
One trajectory of equation \eqref{lang_ring} is demonstrated in Figure
\ref{Langevin} Top Left. The invariant measure satisfies $\rho(X,V) \propto \exp(-\beta (V^2 + U(X)))$ where
$\beta = \frac{2 \gamma}{\sigma^2}.$

Because of the degenerate noise term, we use a modified coupling algorithm 
involving three components which are based on the coupling in~\cite{eberle2019couplings}.  
We consider two realizations $(X^{(1)}, V^{(1)})$
and $(X^{(2)}, V^{(2)})$ of the SDE and note that the difference
process is contractive on the hyperplane 
$Q = X^{(1)}-X^{(2)} + \gamma^{-1} (V^{(1)}-V^{(2)}) = 0$~\cite{eberle2019couplings}.  
We employ reflection coupling when $\|Q\| > 0.08,$ where the reflection tensor is given by $I - 2 * Q Q^T / Q^2.$  When $\|Q\| < 0.08,$
we use synchronized coupling, where both processes use the same realization
of the Brownian noise. The threshold of switching coupling method
(which is $0.08$ in our computation) should be $O(\sqrt{h})$, which is
the distance that $(\hat{X}_{t}, \hat{V}_{t})$ jumps after one step. When the processes are sufficiently close, we
attempt to couple using maximal coupling using two steps of the numerical 
integrator.  Two steps
of the Euler-Maruyama integrator with stepsize $h$ gives 
\begin{equation}
\label{EM_twosteps}
\begin{split}
X_{t+2h} &=    X_t + 2 h V_t - h^2 \nabla U(X_t) - \gamma h V_t + \sigma h^{3/2} N_0, \\
V_{t+2h} &= (1- \gamma h)^2 V_t - (1- \gamma h ) h \nabla U(X_t) - h \nabla U(X_t + h V_t) \\
&\qquad + \sigma (1 - \gamma h) h^{1/2} N_0 + \sigma h^{1/2} N_1,
\end{split}
\end{equation}
where $N_0$ and $N_1$ represent two i.i.d. normal variables.  We sample 
from the above to 
compute the probability of the processes to couple after two steps.
To improve the computational efficiency of the scheme we only test for coupling 
when the processes are close, specifically, when $|X^{(1)} -
X^{(2)}| < 2.5 \sigma h^{3/2}$ and $|V^{(1)} - V^{(2)}| < 2.5 \sigma
h^{1/2}.$

%
%
\begin{figure}[h]
\centerline{\includegraphics[width = 1.1\linewidth]{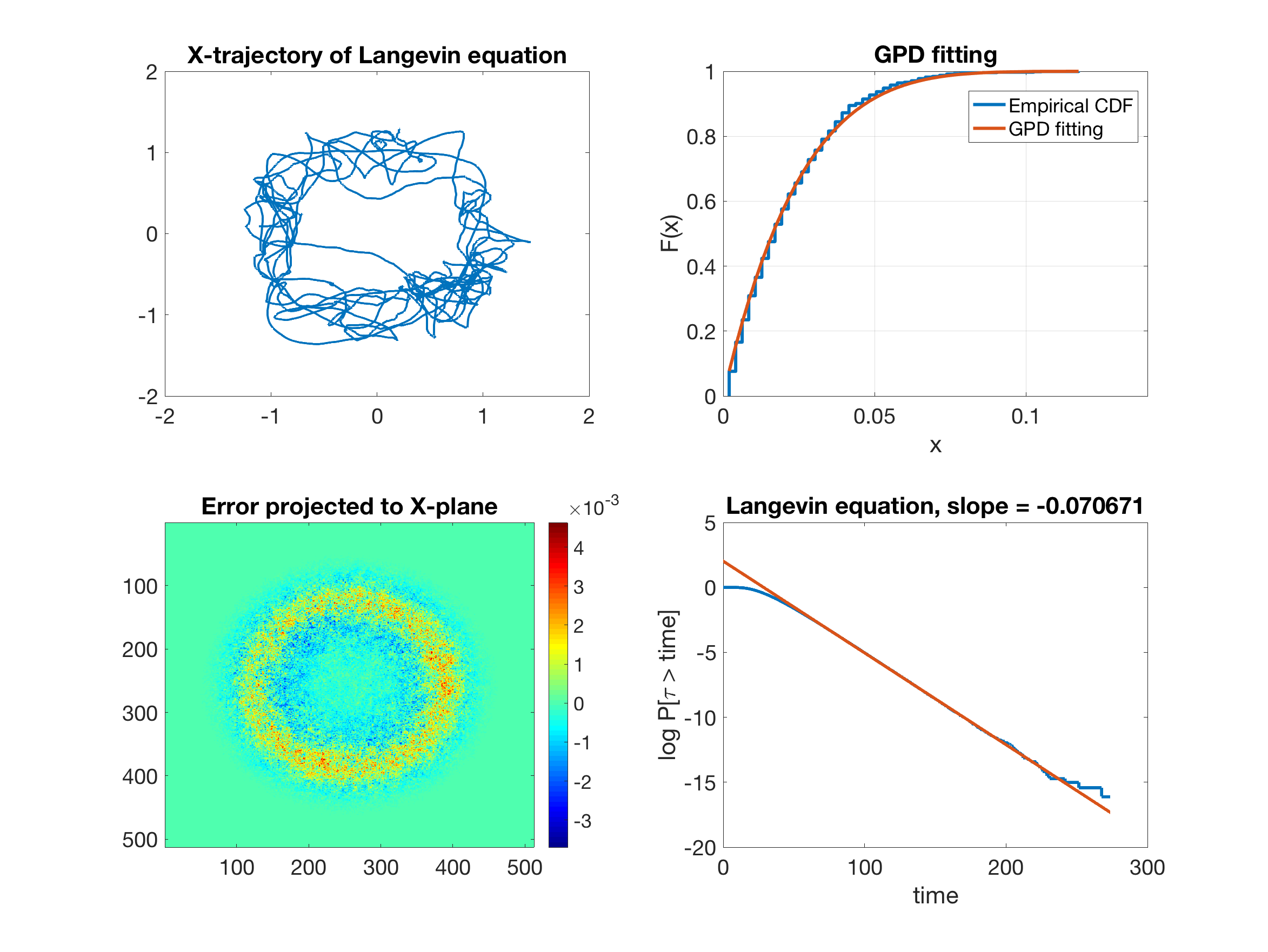}}
\caption{Top left: A sample path of Langevin equation with length
  $100$ (projected to
  X-plane). Top right: Fitting generalized Paredo distribution with $v_{i} =
  1/(1 - r_{i})$. The fitting result is compared with the empirical
  cumulative distribution function. Bottom left: Difference between
  $\hat{\pi}$ and $\pi$ projected to X-plane. Bottom Right: Exponential tail of $\mathbb{P}[\tau_{c} > t]$ versus $t$
  when initial values are uniformly sampled in $\Omega \times \Omega$.
}
\label{Langevin} 
\end{figure}

In this simulation, we choose $\sigma = 0.5$ and truncation time $T =
40.$ The time step size is $h = 0.001$.
Averaged from $8$ long trajectories with length $4\times
10^{6},$ we find the error $d_w(\pi P^T, \pi \hat{P}^T) \leq 0.0111313.$ 
We choose the domain $\Omega = [-3 \,, \,3]^{2} \times [-6 \,, 6]^{2}$ for
the coordinates $(x_{1}, x_{2}, v_{1}, v_{2})$. 
GPD fitting used
$20,000$ pairs of initial values and $1,000$ trajectories for each pair of
initial value giving $\alpha_{\Omega} = 0.3727.$ See Figure
\ref{Langevin} Top Right for a comparison of cumulative distribution
function of the GPD fitting. This gives an upper bound
$$
  \mathrm{d}_{w}(\pi, \hat{\pi}) \leq 0.017745 \,.
$$

It is not easy to numerically estimate the
distance between $\pi$ and $\hat{\pi}$ as they are probability
measures in $\mathbb{R}^{4}$. Instead, we project them to the $X$-plane
and compare the projected probability density functions. Note that the
difference between two projected probability density functions is smaller than that of
$\pi$ and $\hat{\pi}$. In Figure \ref{Langevin} Bottom Left, we can
see the difference between $P_{x}\hat{\pi}$ and $P_{x}\pi$, where $P_{x}$ is the projection 
operator to the $X$-plane. The approximate numerical invariant measure $\hat{\pi}$ in Figure~\ref{Langevin}
is obtained from $80$ long trajectories, each of which is
integrated up to $T = 10^{7}$. The probability density function of
$\hat{\pi}$ is computed on a $512 \times 512$ grid.

Finally, we compute the exponential tail of the coupling time to
obtain the slope $\gamma = -0.07067$. Therefore, when $T = 40$, equation
\eqref{rough} gives a rougher estimate 
$$
  \mathrm{d}_{w}(\pi, \hat{\pi}) \approx 0.011832 \,.
$$

\subsection{Lorenz 96 model}
In this subsection we study a highly chaotic example. Consider equation
\begin{align}
  \label{Lorenz}
\mathrm{d}X^{1}_{t} & = (X^{2}_{t} - X^{D-1}_{t})X^{D}_{t} - X^{1}_{t} + F +
  \sigma \mathrm{d}W^{1}_{t} \\\nonumber
\mathrm{d}X^{2}_{t} & = (X^{3}_{t} - X^{D}_{t})X^{1}_{t} - X^{2}_{t} + F +
  \sigma \mathrm{d}W^{2}_{t} \\\nonumber
&\vdots \\\nonumber
\mathrm{d}X^{i}_{t} & = (X^{i+1}_{t} - X^{i-2}_{t})X^{i-1}_{t} - X^{i}_{t} + F +
  \sigma \mathrm{d}W^{i}_{t} , \quad i = 3, \cdots, N-1\\\nonumber
&\vdots \\\nonumber
\mathrm{d}X^{D}_{t} & = (X^{1}_{t} - X^{D-2}_{t})X^{D-1}_{t} - X^{D}_{t} + F +
  \sigma \mathrm{d}W^{D}_{t}  \,,
\end{align}
where the forcing term $F$ is usually chosen to be $8$. When $D =
4$, the system has a large periodic orbits. It demonstrates chaotic
dynamics when $D \geq 5$ \cite{karimi2010extensive}.  See Figure~\ref{EVTLorenz}~Left 
for the trajectory of the first three variables as an
example.

In our simulations, we use Euler-Mayurama scheme with step size
$h = 0.0001$ to simulate numerical trajectories $\hat{X}_{t}$. Model
parameters are $\sigma = 3$ and $F = 8$. The time
span is chosen to be $T = 3$. When $D = 4$, in Algorithm \ref{finitetime}, we run
$8$ long trajectories with length $3\times 10^{5}$ each to compare
the difference between $\hat{X}_{T}$ and $\hat{X}^{2h}_{T}$. The simulation gives an upper bound
$$
  \mathrm{d}_{w}(\pi P^{T}, \pi \hat{P}^{T}) \leq 0.144864 \,.
$$
Although we have chosen a small time step size $h$, this upper bound is
still relatively large. The error gets even larger when $D = 5$ is
used, because the deterministic dynamics is intensively chaotic. The output of
Algorithm \ref{finitetime} is $0.11946$ for $h = 0.00001$ and
$0.431059$ for $h = 0.0001$. 

Then we run Algorithm \ref{coupling} for $D = 4$ to get the contraction rate of
$\hat{P}^{T}$ for $T = 3$. $\Omega$ is chosen to be the 4D-box $[-16,
\, 19]^{4}$ because when running Algorithm \ref{finitetime}, no
trajectory has ever been outside of this box. The number of initial values $(x_{i}, y_{i})$ is
$20000$. We run $1000$ pairs of trajectories from each initial points
to get $\mathbb{P}[\tau_{c} > T]$. This gives 
$$
  r_{i} = \frac{ \mathbb{P}_{x_{i}, y_{i}}[ \tau_{c} > T]}{ d(x_{i},
    y_{i})} \geq \frac{\mathrm{d}_{w}( \delta_{x_{i}}\hat{P}^{T},
    \delta_{y_{i}}\hat{P}^{T})}{d(x_{i}, y_{i})} 
$$
for $i = 1, \cdots, 20000$. 
The GPD fitting gives $\alpha_{\Omega} = 0.7081$. See Figure \ref{EVTLorenz} Right for the
fitting result. Combine the output of two algorithms, we have the
bound
$$
  \mathrm{d}_{w}(\pi, \hat{\pi}) \leq 0.4963 \,.
$$

\begin{figure}[htbp]
\centerline{\includegraphics[width = \linewidth]{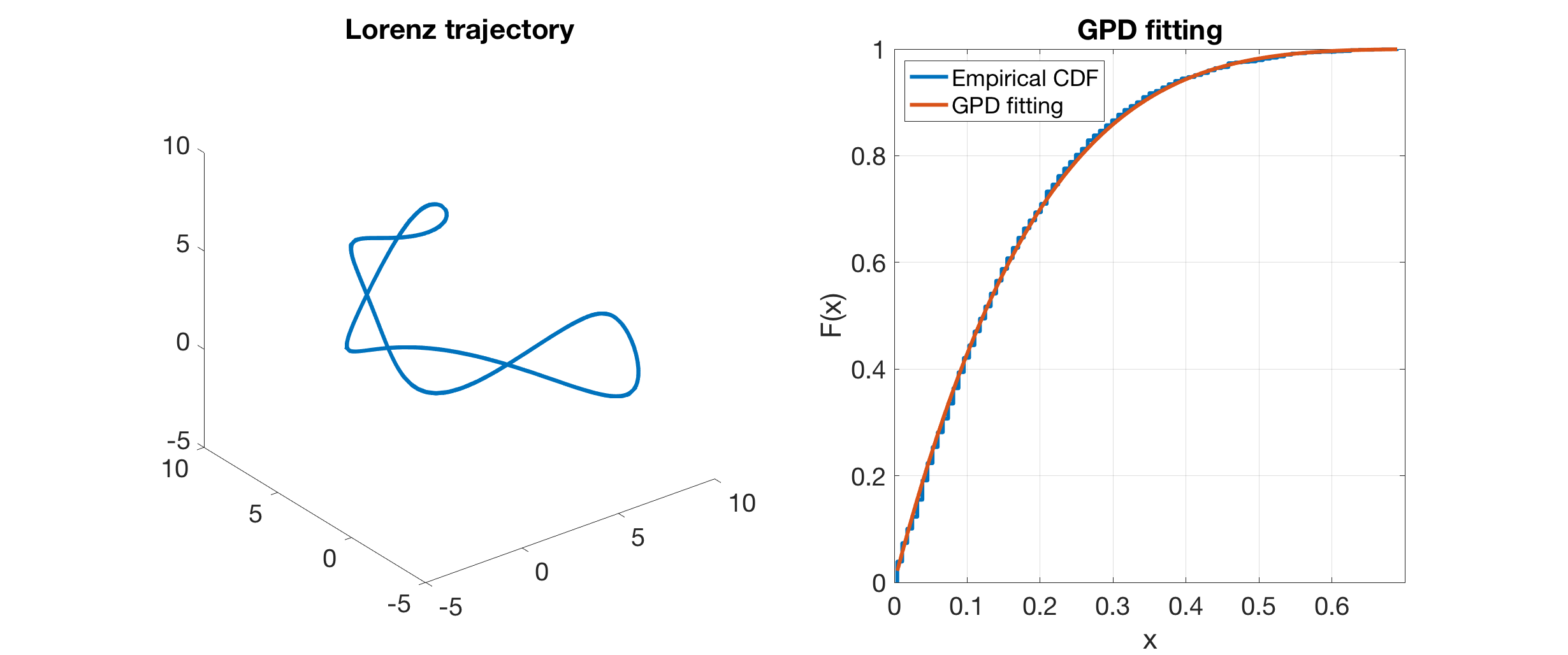}}
\caption{Left: A plot of the first 3 variables of the limit
  cycle. Right: GPD fitting of $\{1/(1 -
r_{i})\}_{i = 1}^{20000}$ and a comparison with the empirical
cumulative distribution function. }
\label{EVTLorenz}
\end{figure}

When $D = 5$ and $h = 1 \times 10^{-5}$, the computational cost of Algorithm
\ref{coupling} becomes very high due to extremely small time step
size. Instead, we compute exponential tails of the coupling time for
$D = 4$ and $D = 5$ with $h = 0.0001$. The result is demonstrated in
Figure \ref{Lorenztail}. We can see that when $D = 5$,
we have an exponential tail $\gamma = 0.12541$. Therefore, if $T = 3$
is unchanged, we have $(1 - e^{-\gamma T})^{-1} = 3.1892$. We conclude
that when $D = 5$, the 1-Wasserstein distance between 
$\pi$ and $\hat{\pi}$ is unacceptably large even if $h = 1 \times
10^{-5}$. This is mainly caused by very large finite time error. In
order to approximate $\pi$ effectively, high order approximation of
equation \eqref{Lorenz} is necessary. 

\begin{figure}[htbp]
\centerline{\includegraphics[width = \linewidth]{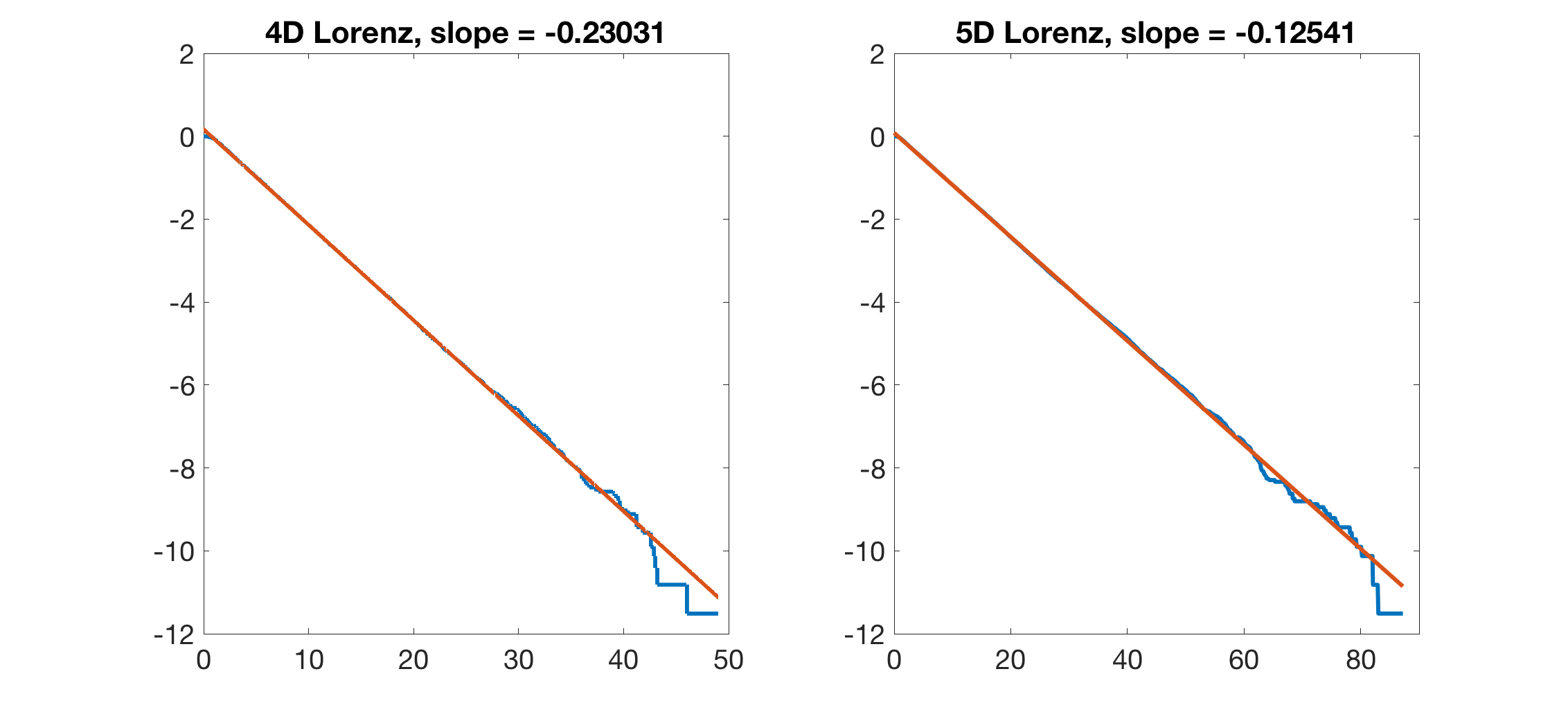}}
\caption{Exponential tail of $\mathbb{P}[\tau_{c} > t]$ versus $t$
  when initial values are uniformly sampled in $\Omega \times
  \Omega$. Left: 4D Lorenz 96 system. Right: 5D Lorenz 96 system.}
\label{Lorenztail}
\end{figure}

\subsection{Stochastically coupled FitzHugh-Nagumo oscillator with mean-field
  interaction}
	\label{sec:fitzhugh}
We consider here a high dimensional example, a
stochastically coupled FitzHugh-Nagumo(FHN) oscillator. FitzHugh-Nagumo
model is a nonlinear model that models the periodic change of membrane
potential of a spiking neuron under external stimulation. The model is
a 2D system
\begin{eqnarray}
\label{FHN1}
 \mu \mathrm{d}u& = &(u - \frac{1}{3}u^{3} - v) \mathrm{d}t +
                           \sqrt{\mu} \sigma \mathrm{d}W_{t}
  \\\nonumber
\mathrm{d} v &=& (u + a) \mathrm{d}t + \sigma \mathrm{d}W_{t} \,,
\end{eqnarray}
where $u$ is the membrane potential, and $v$ is a recovery
variable. 

When $a = 1.05$, the deterministic system admits a stable fixed
point with a small basin of attraction \cite{chen2019spatial}. A suitable random perturbation can drive this system away
from the basin of attraction and trigger limit cycles intermittently.

In this section we consider $N$ coupled equations \eqref{FHN1} with
both nearest-neighbor interaction and mean-field interaction. Let $v =
\sqrt{\mu}v$ be the new recovery variable. We have
\begin{align}
\label{FHNc}
 \mathrm{d}u_{i}& = \left(\frac{1}{\mu}u - \frac{1}{3\mu}u^{3} -
                  \frac{1}{\sqrt{\mu}}v + \frac{d_{u}}{\mu} ( u_{i+1}
                      + u_{i-1} - 2 u_{i}) + \frac{w}{\mu}( \bar{u} - u_{i})\right ) \mathrm{d}t +
                           \frac{\sigma}{\sqrt{\mu}} \mathrm{d}W_{t}
  \\\nonumber
\mathrm{d} v_{i} &=(\frac{1}{\sqrt{\mu}}u + \frac{a}{\sqrt{\mu}}) \mathrm{d}t + \frac{\sigma}{\sqrt{\mu}}\mathrm{d}W_{t} \,,
\end{align}
for $i = 1,\cdots, N$, where $d_{u}$ is the neareast-neighbor coupling
strength, $w$ is the mean field coupling strength, and
$$
  \bar{u} = \frac{1}{N} \sum_{i = 1}^{N} u_{i} 
$$
is the mean membrane potential. In our simulations, we let $u_{-1} =
u_{N}$ and $u_{N+1} = u_{1}$. In other words, $N$ neurons are
connected as a ring. 

The parameters we choose are $d_{u} = 0.03$ and $w = 0.3$. In addition
we have $\sigma = 0.6$. Activities
of neurons are weakly coupled under this parameter set. See Figure
\ref{FHNdynamics} for the dynamics of this system. In particular, from
Figure \ref{FHNdynamics} Right, we can see that nearest neighbor
neurons tend to spike together. However, the global dynamics is only
weakly synchronized. The
dimensions of system in our study are chosen to be $N = 2$ and $N = 40$,
corresponding to stochastic differential equations in $\mathbb{R}^{4}$
and $\mathbb{R}^{80}$. The numerical scheme in our simulation is
Euler-Maruyama scheme with $h = 0.0005$. The finite time span is $T =
3$ for both cases.

\begin{figure}[htbp]
\centerline{\includegraphics[width = \linewidth]{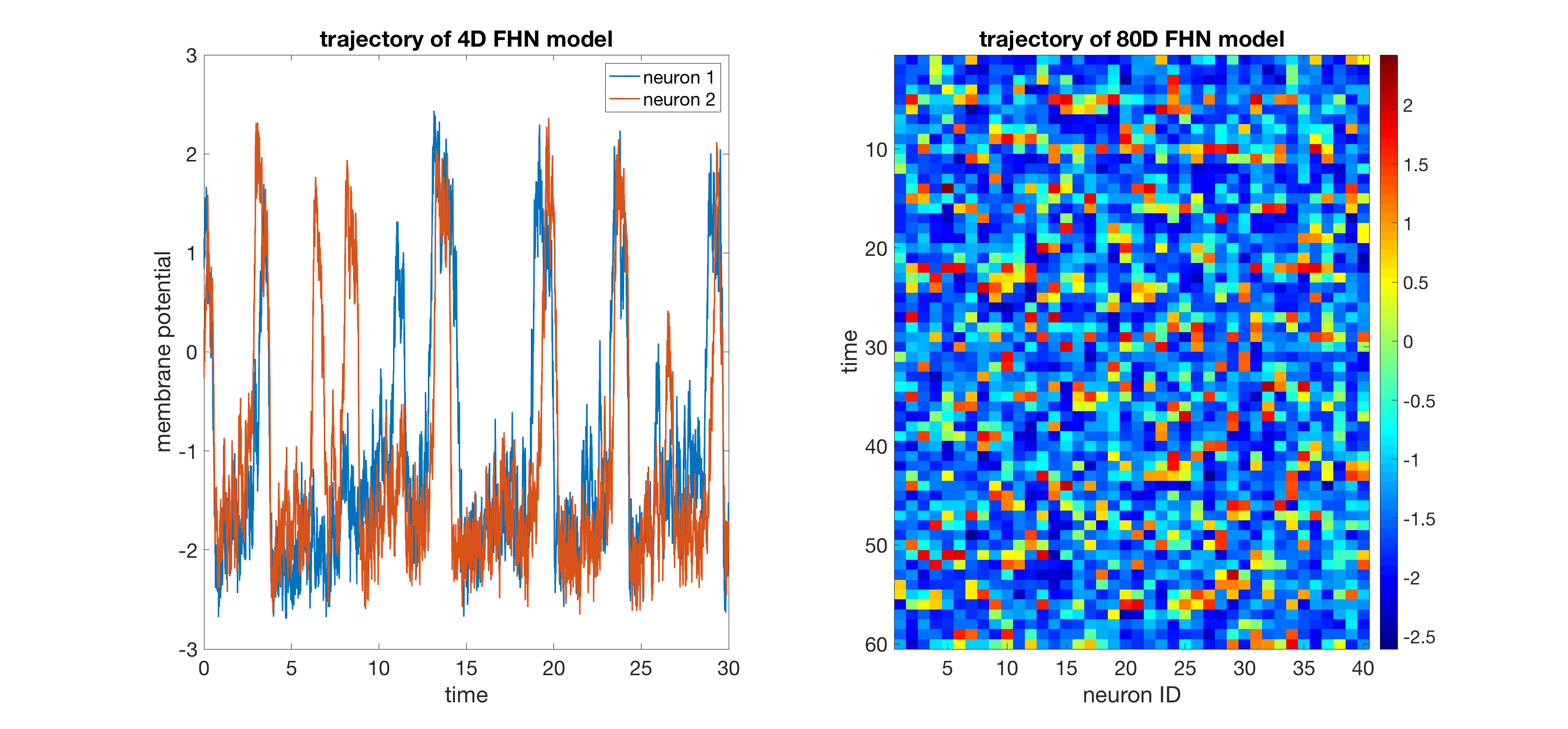}}
\caption{Dynamics of the FitzHugh-Nagumo model. Left: Trajectory of
  membrane potential of two neurons evolving along equation
  \eqref{FHNc} with $N = 2$ and other parameters specified in this paper. Right: Snapshots
  of membrane potential of $40$ neurons evolving along equation
  \eqref{FHNc} with $N = 40$ and other parameters specified in this
  paper.}
\label{FHNdynamics}
\end{figure}

We first run Algorithm \ref{finitetime} and Algorithm \ref{coupling}
for equation \eqref{FHNc} with $N = 2$. In Algorithm \ref{finitetime},
we run $8$ long trajectories up to $3 \times 10^{5}$. The simulation
gives an upper bound
$$
  \mathrm{d}_{w}(\pi P^{T} ,\pi \hat{P}^{T}) \leq 0.0105652
$$
for $T = 3$. Then we run Algorithm \ref{coupling} for $\Omega = [-6, \, 6]^{2N}$ to
get the contraction rate of $\hat{P}^{T}$. The number of initial
values is $40000$. This gives $40000$ coupling probabilities $r_{1},
\cdots, r_{40000}$. Fitting these numbers with generalized Pareto distribution gives
$\alpha_{\Omega} = 0.5197$. See the result in Figure Figure \ref{FHN2}
Left. Combine the output of two algorithms, we
have
$$
  \mathrm{d}_{w}(\pi, \hat{\pi}) \leq 0.0220 \,.
$$
Hence the invariant probability measure simulated by running the
Euler-Maruyama scheme is trustworthy in spite of the presence of
slow-fast dynamics. In addition, we compute the tail of coupling time
for $N = 2$, which is demonstrated in Figure \ref{FHN2} middle. The
exponential tail has a slope $\gamma  = 0.50741$. Therefore, equation
\eqref{rough} gives an estimate 
$$
  \mathrm{d}_{w}(\pi, \hat{\pi}) \approx 0.01351 \,.
$$

When $N = 40$, we still run Algorithm \ref{finitetime} with $8$ long
trajectories up to time $3\times 10^{5}$. This gives us an estimate
$$
  \mathrm{d}_{w}(\pi P^{T} ,\pi \hat{P}^{T}) \leq 0.0443737
$$
for $T = 3$. However, Algorithm \ref{coupling} becomes expensive for
$N = 40$. Instead we compute the exponential of coupling time to get a
rough estimate. The exponential tail of coupling time is demonstrated
in Figure \ref{FHN2} Right. We have an exponential tail with slope
$\gamma = 0.31612$. Therefore, equation \eqref{rough} gives a rough
estimate
$$
  \mathrm{d}_{w}(\pi, \hat{\pi}) \approx 0.07243 \,.
$$
Therefore, we conclude that $\hat{\pi}$ is an acceptable approximation
of $\pi$ when $N = 40$.

\begin{figure}[htbp]
\centerline{\includegraphics[width = 1.25\linewidth]{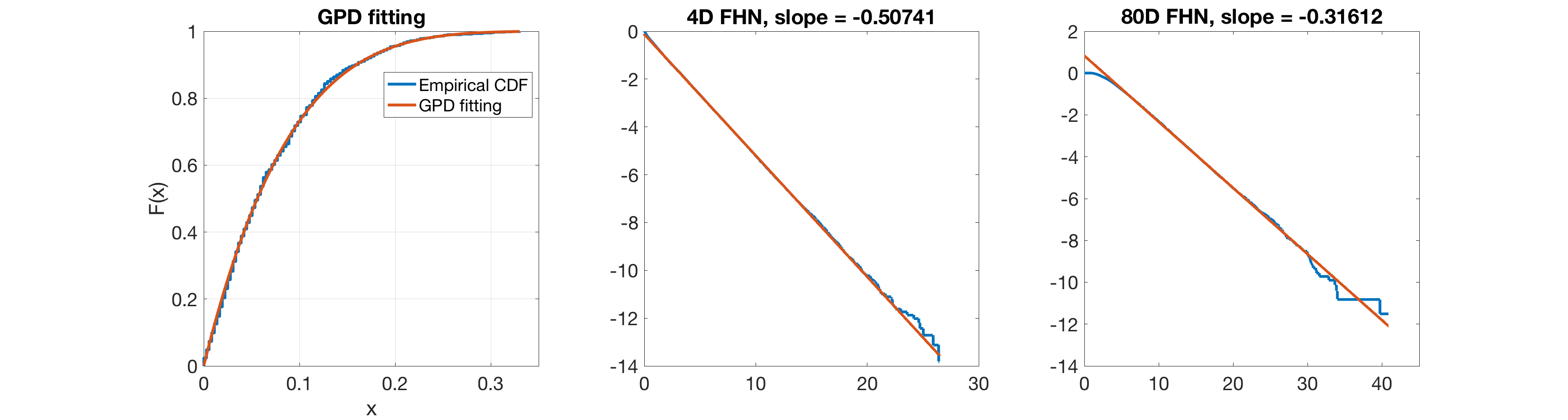}}
\caption{Left: GPD fitting of $\{1/(1 -
r_{i})\}_{i = 1}^{40000}$ and a comparison with the empirical
cumulative distribution function. Middle: Exponential tail of
$\mathbb{P}[\tau_{c} > t]$ versus $t$ for equation \eqref{FHNc} with $N = 2$. Right: Exponential tail of
$\mathbb{P}[\tau_{c} > t]$ versus $t$ for equation \eqref{FHNc} with $N = 40$. }
\label{FHN2}
\end{figure}
\section{Conclusion}
In this paper we provide a coupling-based approach to quantitatively
estimate the distance between the invariant probability measure $\pi$
of a stochastic differential equation (SDE) and that of its numerical
scheme, denoted by $\hat{\pi}$. The key idea is that the distance
$d(\pi, \hat{\pi})$ can be bounded by $\epsilon(1 - \alpha)^{2}$,
where $\epsilon$ is the finite time truncation error
over the time interval $[0, T]$, and $\alpha$
is the rate of contraction of $\hat{P}^{T}$, the time-$T$ transition kernel of
the numerical scheme for the SDE. The finite time truncation error
comes from extrapolation analysis, and we use coupling method to
estimate $\alpha$. Neither of these two estimates relies on spatial
discretization.  Hence our approach is relatively dimension
free. Depending on the practical requirement, we provide one algorithm for
computing a quantitative upper bound of $d(\pi, \hat{\pi})$, and an
efficient algorithm for a ``rough estimate'' of $d(\pi, 
\hat{\pi})$. The performance of these two algorithms are tested with
several numerical examples. Our approach can be extended to
other stochastic processes, such as stochastic differential equations with random
switching and applications related to Hamiltonian Monte Carlo
\cite{bou2018coupling, mao2006stochastic}. Also, our method is not
limited to numerical analysis. In fact, $\hat{\pi}$ can be
the invariant probability measure of any small perturbations of the
original SDE. In this case, our method gives the sensitivity of $\pi$
against small perturbations.

This paper mainly estimates 1-Wasserstein distance between $\pi$ and
$\hat{\pi}$. However, the coupling method can also be used to estimate
the other type of distances, such as the total variation distance. In
addition, in many cases, we are actually more interested in the
error of the expectation of a certain observable when
integrating with respect to $\hat{\pi}$ versus $\pi$. We choose
1-Wasserstein distance mainly because it is more convenient to
estimate the finite time truncation error in 1-Wasserstein
distance.  In fact, there is only a small literature about
estimating finite time truncation error in the total variation norm. 

The difficulty of estimating finite time truncation error in total
variation distance is partially solved if the grid-based SDE solver
introduced in \cite{bou2018continuous} is used. It is much easier to count samples
on grids than in continuous state space. In fact, we find that this grid-based SDE
solver is more compatible with both Fokker-Planck solver in
\cite{li2018data,dobson2019efficient} and the sample quality checking algorithm studied in this paper. In
the future, we will write a separate paper to discuss the application
this sample quality checking algorithm to this grid-based SDE solver.


\bibliography{myref}
\bibliographystyle{amsplain}
\end{document}